\documentclass[10pt,reqno]{amsart} 
\usepackage{amssymb}
\usepackage[cp1250]{inputenc}   
\usepackage{hyperref}
\usepackage[all]{xy}            
\usepackage{verbatim}
\usepackage{tikz}
    \usetikzlibrary{shadows,matrix,arrows}
\usepackage{mathtools}
\usepackage{longtable}
\usepackage{stmaryrd} 
\newcommand*\ccircled[1]{\tikz[baseline=(C.base)]\node[draw,rectangle,rounded corners,inner sep=1.9pt,line width=0.15mm](C){\normalsize#1};\!} 
\newcommand{\triangleu}{\begin{tikzpicture}[line width={0.04em},scale=0.01\baselineskip]
                            \draw (0,0) -- (1,0) -- (1,-1) -- cycle;
                        \end{tikzpicture}}
\newcommand{\corneru}{\begin{tikzpicture}[line width={0.04em},scale=0.01\baselineskip]
                            \draw (0,0) -- (1,0) -- (1,-1);
                        \end{tikzpicture}}

\newtheorem{Thm}{Theorem}[section]
\newtheorem{Prp}[Thm]{Proposition}

\newtheorem{Cnj}[Thm]{Conjecture}

\theoremstyle{definition}

\newtheorem{Exp}[Thm]{Example}

\theoremstyle{remark}
\newtheorem{Rmk}[Thm]{Remarks}
\numberwithin{equation}{section}

\DeclareMathOperator*{\bigast}{\raisebox{-0.6ex}{\scalebox{2.2}{$\ast$}}}
\def\id{\mathrm{id}}        
\def\Ker{\mathrm{Ker}}
\def\Im{\mathrm{Im}}

\def\N{\mathbb{N}}
\def\Z{\mathbb{Z}}
\def\Q{\mathbb{Q}}
\def\R{\mathbb{R}}
\def\C{\mathbb{C}}
\def\B{\mathbb{B}}
\def\S{\mathbb{S}}

\newcommand{\icases}[7]{#7\{\!\begin{smallmatrix}
                                #5#1\hfill;  & ~#5#2\hfill &\!\!\\[#6]
                                #5#3\hfill;  & ~#5#4\hfill &\!\!
                              \end{smallmatrix}}

\newcommand{\rd}[1]{\textcolor[rgb]{1.00,0.00,0.00}{#1}}
\newcommand{\gn}[1]{\textcolor[rgb]{0.00,1.00,0.00}{#1}}
\newcommand{\bl}[1]{\textcolor[rgb]{0.00,0.00,1.00}{#1}}

\begin{document}
\title{(Co)Homology of Lie Algebras\\ via Algebraic Morse Theory}
\author{Leon Lampret}
\address{Institute for Mathematics, Physics and Mechanics, Ljubljana, Slovenia\vspace{-6pt}}
\address{Faculty of Mathematics and Physics, Department of Mathematics, University of Ljubljana, Slovenia}
\email{leon.lampret@imfm.si}
\author{Aleš Vavpetič}
\address{Faculty of Mathematics and Physics, Department of Mathematics, University of Ljubljana, Slovenia}
\email{ales.vavpetic@fmf.uni-lj.si}
\date{March 18, 2016}

\begin{abstract}
The fundamental theorem of cancellation AMT \cite{citearticleJollenbeckADMTACA} and \cite{citearticleSkoldbergMTFAV}, which is the algebraic generalization of discrete Morse theory \cite{citearticleFormanMTCC} for simplicial complexes and smooth Morse theory \cite{citearticleMorseRBCPRFIV} for differentiable manifolds, is discussed in the context of general chain complexes of free modules.
\par The Chevalley (co)homology table of a Lie algebra is often a tremendous beast. Using AMT, we compute the homology of the Lie algebra of all triangular matrices $\frak{sol}_n$ over $\Q$ or $\Z_p$ for large enough primes $p$. We determine the column and row in the table of $H_k(\frak{sol}_n;\Z)$ where the $p$-torsion first appears. 
Module $H_k(\frak{sol}_n;\Z_p)$ is expressed by the homology of a chain subcomplex for the Lie algebra of all strictly triangular matrices $\frak{nil}_n$, using the K\"{u}nneth formula. All conclusions are accompanied by computer experiments.
\par Then we generalize some results to Lie algebras of (strictly) triangular matrices \smash{$\frak{gl}_n^\prec$} and \smash{$\frak{gl}_n^\preceq$} with respect to any partial ordering $\preceq$ on $[n]$. We determine the multiplicative structure of \smash{$H^\ast(\frak{gl}_n^\preceq)$} w.r.t. the cup product over fields of zero or sufficiently large characteristic, the result being the exterior algebra.
\par Matchings used here can be analogously defined for other Lie algebra families and in other (co)homology theories; we collectively call them \emph{normalization matchings}. They are useful for theoretical as well as computational purposes.\vspace{-6mm}
\end{abstract}

\maketitle

\subsection*{Background} Algebraic Morse theory, abbreviated by AMT, is a combinatorial approach to homological algebra, designed for dealing with (co)chain complexes. In principle, it can be applied to any chain complex of free modules, no matter its origin, though finding a meaningful Morse matching is often difficult. The theory was discovered independently in the article \cite{citearticleSkoldbergMTFAV} and thesis \cite{citearticleJollenbeckADMTACA}, building on the ideas of Forman's \cite{citearticleFormanMTCC} in a much broader scope.
\par There are many applications of Forman's DMT, but using AMT happened more rarely up to now. Sk\"{o}ldberg used AMT to compute the Chevalley homology $H_\ast(\frak{h}_n;R)$ of the Heisenberg Lie algebra family over any ring of characteristic $2$ in \cite{citearticleSkoldbergHHLAFCT}. J\"{o}llenbeck applied AMT to obtain the Hochschild homology $HH_\ast(A;R)$ of a very general family of commutative and noncommutative $R$-algebras in \cite{citearticleJollenbeckWelkerMRADMT, citearticleJollenbeckWelkerRRCFADMT}.
\par The theory of AMT lies at the intersection of mathematical fields of homological algebra, algebraic topology, algebraic combinatorics, and commutative algebra.\vspace{1mm}

\subsection*{Motivation} Just as power-sets $2^{[n]}$ are the `largest' finite posets, and symmetric groups $S_n$ are the `largest' finite groups, and matrix algebras $M_n(K)$ are the `largest' finite-dimensional associative algebras, so are arbitrary/triangular/strictly-triangular matrix algebras $\frak{gl}_n(K),\frak{sol}_n(K),\frak{nil}_n(K)$ the `largest' finite-dimensional arbitrary/solvable/nilpotent Lie algebras, and therefore of great significance. But when the field $K$ has nonzero characteristic, the homology of these three families of Lie algebras is completely unknown. Our goal here is investigating the strange patterns in the homology table \ref{3.table} for $\frak{sol}_n$ and for its poset subalgebras $\frak{gl}_n^\preceq$.

\subsection*{Results} In this article, we define a class of Morse matchings \ref{3.matching} \ref{5.matching}, which drastically reduce the size of chain complexes without changing the boundary; they are easily implementable. By inspecting the unmatched basis elements, we prove the following assertions. Let $\preceq$ be any partial ordering of $[n]$. Let $\frak{gl}_n^\preceq$ and $\frak{gl}_n^\prec$ be Lie subalgebras of $\frak{gl}_n$ with bases $\{e_{ij}; i\!\preceq\!j\}$ and $\{e_{ij};i\!\prec\!j\}$ respectively.
\par If $K$ is a field of characteristic $0$ or $p\!\geq\!n$, then \ref{5.cup} $H^\ast(\frak{gl}^\preceq_n;K)\cong \Lambda_K[x_1,\ldots,x_n]$ as graded algebras. If $p\!>\!\frac{k+1}{2}$, then \ref{5.Q} $H_k(\frak{gl}^\preceq_n;\Z)$ does not contain $p$-torsion. By \ref{5.LowDegrees}, $H_1(\frak{gl}_n^\preceq;\Z)\!\cong\!\Z^n$\!, $H_2(\frak{gl}_n^\preceq;\Z)\!\cong\!\Z^{\binom{n}{2}}$\!, $H_3(\frak{gl}_n^\preceq;\Z) \!\cong\! \Z^{\binom{n}{3}}\!\oplus\!\Z_2^{|\{a,c\in[n];\,\exists b: a\prec b\prec c\}|}$\!. There holds row growth \ref{5.RowGrowth} $H_k(\frak{gl}^\preceq_n)\!\cong\! H_k(\frak{gl}^\preceq_{n\!-\!1})\!\oplus\!\ldots$. For every interval in $([n],\preceq)$ with $t\!+\!1$ elements, \ref{5.torsion} $H_{2t\!-\!3}(\frak{gl}^\preceq_n;\Z)$ has a direct summand $\Z_t$.
\par If $\preceq$ is the usual total order on $[n]$, so that $\frak{gl}_n^\preceq\!=\!\frak{sol}_n$ is the solvable Lie algebra of all upper triangular matrices with rank $N\!=\!\binom{n+1}{2}$, then \ref{3.Zp} $p$-torsion in the homology table \ref{3.table} appears for the first time in $p\!+\!1$-th column and $2p\!-\!1$-th row. Homological dimension \ref{3.HighDegrees} of $\frak{sol}_n$ over $\Q$ is $n$, over $\Z_2$ is $\icases{N}{\text{if }n\text{ is odd}}{N-n/2}{\text{if }n\text{ is even}}{\scriptstyle}{-1pt}{\big}$, over $\Z_3$ is $N\!-\!m$ when $n\!=\!3m\!\pm\!1$, and over $\Z_p$ with $p\!\geq\!5$ is smaller than over $\Z_2$.
\par There holds \ref{6.tensor} $C_\ast(\frak{gl}^\preceq_n;\Z_p)\simeq C_{\ast,p}(\frak{gl}^\prec_n;\Z_p)\!\otimes\!C_\ast(\frak{dgn}_n;\Z_p)$ where $\frak{dgn}_n$ is the Lie algebra of diagonals, so $\big\{p\text{-torsion of }H_\ast(\frak{gl}^\preceq_n;\Z)\big\} \!=\! \big\{p\text{-torsion of }H_\ast C_{\ast,p}(\frak{gl}^\preceq_n;\Z)\big\}$. Here $C_{\ast,p}$ denotes the chain subcomplex, spanned by all wedges $v\!=\!\ldots e_{rs}\ldots$, in which for every index $i\!\in\![n]$ the number of its right appearances $i\!=\!r$ in $v$ minus the number of its left appearances $i\!=\!s$ in $v$ is a multiple of $p$. This is a most useful complex: for certain $p$'s and appropriate posets, it is small enough to allow direct computations, with fascinating results, as witnessed in \cite{citearticleLampretVavpeticCPLA} and \cite{citearticleLampretVavpeticTTLAS}.

\subsection*{Conventions} Throughout this article, $R$ will denote a commutative unital ring. Any $R$-module is assumed to be unital. Letter $p$ will always denote a prime number. Notation $R'\!\leq\!R$ means that $R'$ is a unital subring of $R$, e.g. $\Z_p\!\leq\!R$ iff $R$ has characteristic $p$, and $\Q\!\leq\!R$ iff all integer multiples of $1_{\!R}$ are invertible. Also $R^\times$ is the set of units of $R$ (elements that have a multiplicative inverse), and $R^\times\!\cap\!\Z$ is the set of invertible integer multiples of $1_{\!R}$. The free $R$-module on a set $I$ is denoted by $R^{(I)}$. In a digraph (=directed graph), given an edge $v\!\to\!w$, the vertex $v$ is \emph{initial} and $w$ is \emph{terminal}. We denote $[n]\!=\!\{1,2,\ldots,n\}$,\, $\delta_{ij}\!=\! {\tiny  \big\{\!\begin{smallmatrix}   1 & \text{if } i\!=\!j\\[-1pt]
                0 & \text{if } i\!\neq\!j\end{smallmatrix}}$,\, $\N\!=\!\{0,1,2,\ldots\}$. 

\vspace{4mm}
\section{Algebraic Morse Theory}\label{section1}
In this section, we recall the main theorem of AMT in a user-friendly form that is convenient for later reference. The main idea is as follows. Given a chain complex $(C_\ast,\partial_\ast)$ of free modules, construct a weighted directed graph $\Gamma_{C_\ast}$ (vertices are the basis elements of all $C_k$, weights of edges are the nonzero entries of matrices $\partial_k$), then intelligently choose a subset $\mathcal{M}$ of the edges of $\Gamma_{C_\ast}$ which is a matching with nice properties. AMT states that every such edge in $\mathcal{M}$ kills two modules $R$ (which do not contribute to homology): in this way, $\mathcal{M}$ decreases the number of basis elements, but it may also change the boundary $\partial_\ast$. Later, we will see that for $\frak{nil}_n$ the boundary indeed changes, but for $\frak{sol}_n$ it happily stays the same.

\subsection{Formulation}\label{1.1} Suppose we are given a chain complex of $R$-modules $$C_\ast\!=\!(C_k,\partial_k)_{k\in\Z}\!:~~~~ \ldots\longrightarrow\bigoplus_{v\in I_{k+1}}\!C_{k+1,v} \overset{\partial_{k+1}}{\longrightarrow} \bigoplus_{v\in I_k}\!C_{k,v} \overset{\partial_k}{\longrightarrow} \bigoplus_{v\in I_{k\!-\!1}} \!C_{k\!-\!1,v}\longrightarrow\ldots,$$
i.e. index sets $I_k$ are arbitrary, $C_{k,v}$ are any modules and $\partial_k$ are any homomorphisms with the property $\partial_k\partial_{k+1}\!=\!0$ for all $k\!\in\!\Z$. For $u\!\in\!I_k$ and $v\!\in\!I_{k-1}$, let $\partial_{k,u,v}\!=\!\partial_{uv}$ denote the $R$-module morphism $C_{k,u}\smash{\overset{\iota}{\longrightarrow}} C_k\smash{\overset{\partial_k}{\longrightarrow}} C_{k-1}\smash{\overset{\kappa}{\longrightarrow}} C_{k-1,v}$, where $\iota$ is the inclusion and $\kappa$ the projection. The associated \emph{digraph of $C_\ast$}, denoted by $\Gamma_{\!C_\ast}$, is a directed simple graph with vertex set the disjoint union $\bigsqcup_{k\in\Z}\!I_k$ and edge set $\{(u,v);\, \exists k\!\in\!\Z\!:\, i\!\in\!I_k,j\!\in\!I_{k-1},\partial_{k,u,v}\!\neq\!0\}$. We denote each edge $(u,v)$ by $u\!\to\!v$.
\par Let $\mathcal{M}$ be a subset of the set of edges of $\Gamma_{\!C_\ast}$. The associated \emph{digraph of $\mathcal{M}$}, denoted $\Gamma_{\!C_\ast}^\mathcal{M}$, is the digraph $\Gamma_{\!C_\ast}\!$ with all the edges in $\mathcal{M}$ having reversed orientation, $$\big(\{\text{vertices of }\Gamma_{\!C_\ast}\!\},\,\{\text{edges of }\Gamma_{\!C_\ast}\!\text{ not in }\mathcal{M}\}\!\cup\!\{v\!\to\!u;\,u\!\to\!v\!\in\!\mathcal{M}\}\big).$$ We say that $\mathcal{M}$ is a \emph{Morse matching} or \emph{acyclic matching} or \emph{gradient vector field} on $\Gamma_{\!C_\ast}$ when the following conditions hold:
\begin{enumerate}
\item[$(1)$] $\mathcal{M}$ is a \emph{matching}, i.e. edges in $\mathcal{M}$ have no common endpoints, i.e. whenever $u\!\to\!v,u'\!\to\!v'\!\in\!\mathcal{M}$ then $|\{u,v,u',v'\}|\!=\!4$;
\item[$(2)$] for every edge $u\!\to\!v$ in $\mathcal{M}$, the corresponding map $\partial_{uv}$ is an isomorphism;
\item[$(3)$] $\Gamma_{\!C_\ast}^\mathcal{M}$ contains no directed cycles, and for every $k$ there does not exist an infinite path $u_1\!\to\!v_1\!\to\!u_2\!\to\!v_2\!\to\!\ldots$ in $\Gamma_{\!C_\ast}^\mathcal{M}$ with all $u_1,u_2,\ldots\!\in\!I_k$.
\end{enumerate}
\par A vertex in $\Gamma_{\!C_\ast}^\mathcal{M}$ that is not incident to an edge in $\mathcal{M}$ is called \emph{$\mathcal{M}$-critical}. We denote $\mathring{\mathcal{M}} =\!\{\mathcal{M}\text{-critical vertices of }\Gamma_{\!C_\ast}^\mathcal{M}\}$,\, $\mathring{\mathcal{M}}_k= \mathring{\mathcal{M}}\cap I_k$,\, $\mathring{C}_k=\bigoplus_{v\in\mathring{\mathcal{M}}_k}C_{k,v}$.
\par If $u\!\in\!\mathring{\mathcal{M}}_k$ and $v\!\in\!\mathring{\mathcal{M}}_{k\!-\!1}$, then the set of all directed trails in $\Gamma_{C_\ast}^\mathcal{M}$ from $u$ to $v$ with vertices in $I_k\!\cup\!I_{k\!-\!1}$ is denoted by $\Gamma_{u,v}^\mathcal{M}$. By $(1)$ and $(3)$ such `zig-zag' trails are finite paths. If $\gamma = (u\!=\!v_1\!\to\!\ldots\!\to\!v_{2i}\!=\!v) \in \Gamma_{u,v}^\mathcal{M}$, then the \emph{gradient path} along $\gamma$ is the signed composition of morphisms along $\gamma$ that maps from $C_{k,u}$ to $C_{k\!-\!1,v}$, i.e.
$$\partial_\gamma=\partial_{v_{2i\!-\!1}v_{2i}}(-\partial^{-1}_{v_{2i\!-\!1}v_{2i\!-\!2}})\ldots(-\partial^{-1}_{v_5v_4})\partial_{v_3v_4}(-\partial^{-1}_{v_3v_2})\partial_{v_1v_2}.$$
The inverses exist by $(2)$. 
Define $\mathring{\partial}_k\!: \mathring{C}_k\!\rightarrow\!\mathring{C}_{k\!-\!1}$, for $x\!\in\!C_{k,u}$ with $u\!\in\!\mathring{\mathcal{M}}_k$, by 
$$\mathring{\partial}_k(x) = \textstyle{\sum_{v\in \mathring{\mathcal{M}}_{k\!-\!1},\gamma\in\Gamma_{u,v}^\mathcal{M}}\!\partial_\gamma(x)}.$$\vspace{0mm}

\begin{Thm}[\cite{citearticleSkoldbergMTFAV},\cite{citearticleJollenbeckADMTACA},\cite{citearticleKozlovDMTFCC}] Every Morse matching $\mathcal{M}$ on $\Gamma_{C_\ast}$ induces a homotopy equivalence of chain complexes $(C_k,\partial_k)_{k\in\Z} \simeq (\mathring{C}_k,\mathring{\partial}_k)_{k\in\Z}$, hence $H^\ast(C_\ast)\!\cong\!H^\ast(\mathring{C}_\ast)$, which sends $C_{k,u}\!\ni\!x \!\longmapsto\! \sum_{v\in\mathring{\mathcal{M}}_k,\gamma\in\Gamma^{\mathcal{M}}_{u,v}}\!\partial_\gamma(x)$ and $\sum_{u\in I_k,\gamma\in\Gamma^{\mathcal{M}}_{v,u}}\!\partial_\gamma(x) \!\longmapsfrom\! x\!\in\!\mathring{C}_{k,v}$.
\end{Thm}\vspace{1mm}

\begin{Rmk} 
$\bullet$ In most cases of usage, all $R$-modules $C_{k,u}$ will be just $R$, so that morphism $\partial_{uv}$ is multiplication $R\smash{\overset{\cdot r}{\longrightarrow}}R$ for some $r\!\in\!R$, which is called the \emph{weight} of edge $u\!\to\!v$. Then $\partial_{uv}$ is an isomorphism iff $r$ is a unit of $R$, with inverse $R\,\smash{\overset{:r}{\longleftarrow}}R$.\\
$\bullet$ If $\mathring{\mathcal{M}}_{k+1}\!=\!\emptyset\!=\!\mathring{\mathcal{M}}_{k-1}$, so that there are no zig-zag paths (hence $\mathring{\partial}_{k+1}\!=\!0\!=\!\mathring{\partial}_k$), then it just remains to calculate $\mathring{\mathcal{M}}_k$ which gives $H_k\cong\bigoplus_{v\in\mathring{\mathcal{M}}_k}C_{k,v}$.\\
$\bullet$ Empty set $\mathcal{M}\!=\!\{\}$ is always a (useless) Morse matching, with $\mathring{C}_\ast\!=\!C_\ast$ and $\mathring{\partial}_\ast\!=\!\partial_\ast$.\\
$\bullet$ For every edge $u\!\to\!v$ that we add to $\mathcal{M}$, we kill modules $C_{k,u}$ and $C_{k\!-\!1,v}$ without changing the homology of the complex (provided $\mathcal{M}\cup\{u\!\to\!v\}$ is still a Morse matching). Therefore it is usually desirable for $\mathcal{M}$ to be as large as possible.\\
$\bullet$ On the other hand, the more edges that $\mathcal{M}$ contains, the more zig-zag paths there are in $\Gamma_{\!C}^\mathcal{M}$, meaning that the sum $\mathring{\partial}_k$ is large. Thus it may also be desirable that $\mathcal{M}$ does not contain too many edges. Finding a good balance is important.\\
$\bullet$ Notice: Given matchings $\mathcal{M}\!\subseteq\!\mathcal{M'}$ in $\Gamma$, if digraph $\Gamma^\mathcal{M}$ contains a cycle, then so too does $\Gamma^{\mathcal{M}'}$\!. Thus enlarging a no-good matching is pointless.
\end{Rmk}\vspace{2mm}

\subsection{Examples} Before proceeding to more serious examples, we look at some simple cases to demonstrate how to use AMT. Simplicial and Lie algebra (co)homology behave similarly (vertices of $\Gamma_{C_\ast}$ behave as sets), which is in sharp contrast to group and associative algebra homology (there, vertices behave as sequences, which makes defining $\mathcal{M}$ very different). Therefore, we use the good ol' DMT as our first steps.
\par Let $\Delta$ be a finite-dimensional simplicial complex and $\Delta\!^{[k]}\!=\!\{k\text{-simplices of }\Delta\}$. The (co)homology of $\Delta$ is the (co)homology of the Poincar\'{e} chain complex
$$\begin{array}{c}
\ldots\longrightarrow R^{(\Delta\!^{[k]})}\longrightarrow R^{(\Delta\!^{[k\!-\!1]})}\longrightarrow \ldots\longrightarrow R^{(\Delta\!^{[1]})}\longrightarrow R^{(\Delta\!^{[0]})},\\
\partial\{v_0,\ldots,v_k\}=\sum_{i=0}^k(-1)^i\{v_0,\ldots,\widehat{v_i},\ldots,v_k\}.
\end{array}$$
The vertices of $\Gamma_{C_\ast}$ are the simplices of $\Delta$, and for each vertex the corresponding module is isomorphic to $R$. For a $k$-simplex $\sigma$ and $(k\!-\!1)$-simplex $\tau$, the corresponding boundary operator is either zero or an isomorphism: $\partial_{\sigma,\tau}\!=\!\smash{\icases{\pm\id}{\text{if }\tau\subset\sigma}{0}{\text{otherwise}}{\scriptstyle}{-1pt}{\big}}$.\vspace{0mm}

\begin{Exp} Consider the $3$-dimensional simplicial complex $\Delta$ below left. Its simplicial chain complex is depicted below right, together with a chosen Morse matching $\mathcal{M}$ (clearly $(1)$ and $(2)$ hold, but acyclicity $(3)$ must be checked manually). The circled vertices are the $\mathcal{M}$-critical ones.
$$\xymatrix@R=0.3mm@C=5mm{~~~~\Delta&\\ &\\ &\\ \raisebox{-.5\height}{\includegraphics[width=0.17\textwidth]{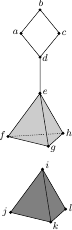}} & }\hspace{2mm}
\xymatrix@R=-1pt@C=5mm{
\smash{R^{(\Delta\!^{[3]}\!)}}\!=R^1 \ar[r]     & \smash{R^{(\Delta\!^{[2]}\!)}}\!=R^8 \ar[r]   & \smash{R^{(\Delta\!^{[1]}\!)}}\!=R^{17} \ar[r]& \smash{R^{(\Delta\!^{[0]}\!)}}\!=R^{12}\\
                                                &                                               & ab\ar[r]\ar|-{\mathcal{M}}[rd]                & \ccircled{$a$}\\
                                                &                                               & ad\ar[ru] \ar[rdd]|-{\mathcal{M}}             & b\\
                                                &                                               & bc\ar[ru] \ar[r]|-{\mathcal{M}}               & c\\
                                                &                                               & \ccircled{$cd$}\ar[ru] \ar[r]               	& d\\
                                                &                                               & de\ar[ru] \ar[rdd]|-{\mathcal{M}}             & \\
                                                &                                               & ef\ar[rd] \ar[rdd]|-{\mathcal{M}}             & \\
                                                & efg \ar[ru] \ar[r] \ar[rdd]|-{\mathcal{M}}    & eg\ar[r]  \ar[rdd]|-{\mathcal{M}}             & e\\
                                                & efh \ar[ruu]\ar[r] \ar[rdd]|-{\mathcal{M}}    & eh\ar[ru] \ar[rdd]|-{\mathcal{M}}             & f\\
                                                & egh \ar[ruu]\ar[ru]\ar[rdd]|-{\mathcal{M}}    & fg\ar[ru] \ar[r]                              & g\\
                                                & \ccircled{$f\!gh$} \ar[ru] \ar[r] \ar[rd]     & fh\ar[ruu]\ar[r]                              & h\\
                                                &                                               & gh\ar[ruu]\ar[ru]                             & \\
                                                &                                               & ij\ar[rd] \ar[rdd]|-{\mathcal{M}}             & \\
                                                & ijk\ar[ru]\ar[r]  \ar[rdd]|-{\mathcal{M}}     & ik\ar[r]  \ar[rdd]|-{\mathcal{M}}             & \ccircled{$i$}\\
 ijkl\ar[ru]\ar[r]\ar[rd]\ar[rdd]|-{\mathcal{M}}& ijl\ar[ruu]\ar[r] \ar[rdd]|-{\mathcal{M}}     & il\ar[ru] \ar[rdd]|-{\mathcal{M}}             & j\\
                                                & ikl\ar[ruu]\ar[ru]\ar[rdd]|-{\mathcal{M}}     & jk\ar[ru] \ar[r]                              & k\\
                                                & jkl\ar[ru]\ar[r]  \ar[rd]                     & jl\ar[ruu]\ar[r]                              & l\\
                                                &                                               & kl\ar[ruu]\ar[ru]                             & \\}$$
\noindent Among critical vertices, there are only two zig-zag paths, namely $\gamma\!: cd\,\smash{\overset{1}{\rightarrow}}\, d\,\smash{\overset{1}{\leftarrow}}\, ad\,\smash{\overset{-\!1}{\rightarrow}}\, a$ and $\gamma'\!\!: cd\,\smash{\overset{-\!1}{\rightarrow}}\, c\,\smash{\overset{1}{\leftarrow}}\, bc\,\smash{\overset{-\!1}{\rightarrow}}\, b\,\smash{\overset{1}{\leftarrow}}\, ab\,\smash{\overset{-\!1}{\rightarrow}}\, a$, hence there holds $\partial_\gamma\!=\! 1(-\frac{1}{1})(-1) \!=\!1$ and $\partial_{\gamma'}\!=\! (-1)(-\frac{1}{1})(-1)(-\frac{1}{1})(-1) \!=\!-1$, so they cancel each other, i.e. $\mathring{\partial}(cd)\!=\!(1\!-\!1)a\!=\!0$.
\par By AMT, this chain complex is homotopy equivalent to $0\, \smash{\xrightarrow{\scriptscriptstyle~~~~0~~~~}} \,R\, \smash{\xrightarrow{\scriptscriptstyle~~~~0~~~~}} \,R\, \smash{\xrightarrow{\scriptscriptstyle[1-1,0]}} \,R^2$, hence the simplicial (co)homology of $\Delta$ is
\begin{equation*} H^3\!\cong\!H_3\!\cong\!0,\;\; H^2\!\cong\!H_2\!\cong\!R,\;\; H^1\!\cong\!H_1\!\cong\!R,\;\; H^0\!\cong\! H_0\!\cong\!R^2.\tag*{$\lozenge$} \end{equation*}
\end{Exp}\vspace{1mm}

\begin{Exp} The following examples are useful exercises, with which the unfamiliar reader should get his hands dirty, to get more accustomed to the combinatorics of AMT. Drawing pictures of particular cases are highly recommended, as well as checking that $\mathcal{M}$ is indeed a Morse matching, and calculating $\mathring{\partial}$ is obligatory.\\[5pt]
$\bullet$ If $\Delta\!=\!\B^n$ is a simplicial $n$-ball and $\mathcal{M}\!=\!\{\sigma\!\cup\!\{0\}\!\to\!\sigma;\, 0\!\notin\!\sigma\}$, then $\mathring{\mathcal{M}}\!=\!\{\{0\}\}$ and $\mathring{\partial}\!=\!0$. Therefore $H_k(\Delta)\!\cong\! H^k(\Delta)\!\cong\! \smash{\icases{R}{\!k=0}{0}{\!\text{otherwise}}{\scriptstyle}{-1pt}{\big}}$.\\[5pt]
$\bullet$ If $\Delta\!=\!\S^n$ is a simplicial $n$-sphere and $\mathcal{M}\!=\!\{\sigma\!\cup\!\{0\}\!\to\!\sigma;\, 0\!\notin\!\sigma\}$, then $\mathring{\mathcal{M}}\!=\!\{\{0\},\{1,\ldots,n\!+\!1\}\}$ and $\mathring{\partial}\!=\!0$. Therefore $H_k(\Delta)\!\cong\! H^k(\Delta)\!\cong\! \smash{\icases{R}{\!k=0,n}{0}{\!\text{otherwise}}{\scriptstyle}{-1pt}{\big}}$.\\[5pt]
$\bullet$ If $\Delta\!=\!\smash{{\B^n}^{[k]}}$ is the $k$-skeleton of an $n$-ball and $\mathcal{M}\!=\!\{\sigma\!\cup\!\{0\}\!\to\!\sigma;\, 0\!\notin\!\sigma\}$, then $\mathring{\mathcal{M}}\!=\! \{\{0\}\}\!\cup\! \{\sigma\!\subseteq\![n]; |\sigma|\!=\!k\!+\!1\}$, so $\widetilde{H}_i(\Delta)\!\cong\! \widetilde{H}^i(\Delta)\!\cong\! \icases{R^N}{\!i=k}{0}{\!i\neq k}{\scriptstyle}{-0pt}{\Big}$ where $N\!=\!\binom{n}{k+1}$.\\[5pt]
$\bullet$ If $\Delta\!=\!\B^1_n$ is an $n$-path and $\mathcal{M}\!=\!\{\{i\!-\!\!1,i\}\!\to\!\{i\};\, i\!=\!1,\ldots,n\}$, then $\mathring{\mathcal{M}}\!=\!\{\{0\}\}$ and $\mathring{\partial}\!=\!0$. Therefore $H_k(\Delta)\!\cong\! H^k(\Delta)\!\cong\! \smash{\icases{R}{\!k=0}{0}{\!\text{otherwise}}{\scriptstyle}{-1pt}{\big}}$.\\[5pt]
$\bullet$ If $\Delta\!=\!\S^1_n$ is an $n$-cycle and $\mathcal{M}\!=\!\{\{i\!-\!\!1,i\}\!\to\!\{i\};\, i\!=\!2,\ldots,n\}$, then $\mathring{\mathcal{M}}\!=\!\{\{1\},\{1,n\}\}$ and $\mathring{\partial}\!=\!0$. Therefore $H_k(\Delta)\!\cong\! H^k(\Delta)\!\cong\! \smash{\icases{R}{\!k=0,1}{0}{\!\text{otherwise}}{\scriptstyle}{-1pt}{\big}}$.\\[5pt]
$\bullet$ Let $\Delta\!=\!\bigvee_{\!i\in I}\!\Delta_i$ be a wedge sum of arbitrarily many simplicial complexes, joined at $0$-simplex $v$ (so that $\Delta^{\![k]}\!=\!\bigsqcup_i\!\Delta_i^{\![k]}$ for $k\!\geq\!1$). In the digraph for the chain complex giving reduced homology of $\Delta$ (just one additional vertex $\emptyset$ in degree $-1$ that connects to all $0$-simplices), define $\mathcal{M}\!=\!\{\{v\}\!\to\!\emptyset\}$. In the digraph for the direct sum of chain complexes giving reduced homology of each $\Delta_i$ ($|I|$ more copies of the vertex $\emptyset$ in degree $-1$ that connects to all $0$-simplices), define $\mathcal{M}\!=\!\{\{v_i\}\!\to\!\emptyset_i; i\!\in\!I\}$. In both cases, the newly obtained complex is the same. This shows that $$\textstyle{\widetilde{H}_k(\Delta)\cong \bigoplus_i\widetilde{H}_k(\Delta_i)\text{ \;\;\;\;and\;\;\;\; } \widetilde{H}^k(\Delta)\cong \prod_i\widetilde{H}^k(\Delta_i).}$$
$\bullet$ Let $\Delta\!=\!\Delta^{\![0]}\!\cup\!\Delta^{\![1]}\!=\!\binom{[n]}{k}\!\cup\!\{\{u,v\}; u\!\cap\!v\!=\!\emptyset\}$ be the $(n,\!k)$-th Kneser graph (which is a $1$-dimensional simplicial complex, vertices are the $k$-element subsets of $[n]$, edges correspond to pairs of subsets with empty intersection). Denote $m\!=\!\binom{n}{k}$. If $2k\!>\!n$, then $\Delta\!=\!\bigsqcup_{i=1}^m\!\mathrm{pt}$. If $2k\!=\!n$, then $\Delta\!\simeq\!\bigsqcup_{i=1}^{m/2}\!\mathrm{pt}$. But if $2k\!<\!n$, define $\varphi\!:\!\binom{[n]}{k}\!\to\!\binom{[n]}{k}$ $$\varphi\!: u\!\mapsto\!\icases{\text{first }k\text{ elements of }[n]\setminus u}{\!u\nsubseteq[2k]}{[2k]\setminus u\text{ without minimum but with }2k+1}{\!u\subseteq[2k]}{\scriptstyle}{2pt}{\Big}$$ and let $\mathcal{M}\!=\!\big\{\{u,\varphi(u)\}\!\to\!\{u\};\,u\!\neq\![k]\big\}$. Then $\forall u\exists i\!: \varphi^i(u)\!=\![k]$ (hence there are no cycles in $\Gamma_{C_\ast}^{\mathcal{M}}$) and $\mathring{\mathcal{M}}\!=\!\big\{\{[k]\},\{u,v\}; \varphi(u)\!\neq\!v\big\}$ and finally $\Delta\simeq\bigvee_{\!i=1}^N\!\S^1$ where $N\!=\!\binom{n}{k}\!\binom{n\!-\!k}{k}\!/2\!-\!\binom{n}{k}\!+\!1$. Recall that any connected graph with $v$ vertices and $e$ edges is homotopy equivalent to a wedge of $e\!-\!v\!+\!1$ copies of $\S^1$ (contract a maximal tree). But here we obtain an explicit spanning tree in $\Delta$, namely $\{\{u,\varphi(u)\};\,u\!\neq\![k]\big\}$.\\[5pt]
$\bullet$ Let $\B_n^1$ be an $n$-path (vertices $0,1,\ldots,n$, edges $\{0,1\},\{1,2\},\ldots,\{n\!-\!1,n\}$). Let $\Delta$ be the anticlique complex of $\B_n^1$ (vertices of $\Delta$ are the vertices of $\B_n^1$, simplices of $\Delta$ are the sets of vertices of $\B_n^1$ which contain no edges). Define a matching by $\mathcal{M}\!=\! \big\{\sigma\!\to\!\sigma\!\setminus\!\{3i\}; \min\sigma\!\cap\!(3\N\!+\!2)\!=\!3i\!+\!2,3i\!\in\!\sigma\big\}$ so that $\smash{\mathring{\mathcal{M}}}\!=\!\{\sigma; \sigma\!\cap\!(3\N\!+\!2)\!=\!\emptyset\}$. \vspace{-5pt}
\begin{longtable}[c]{ll}
    $n\!=\!6$:  & $\xymatrix@R=-1pt@C=5mm{\overset{\scriptscriptstyle0}{\circ}\!\!\ar@{-}[r]& \!\!\!\!\overset{\scriptscriptstyle1}{\circ}\!\!\!\!\ar@{-}[r]& \!\!\!\!\!\overset{\scriptscriptstyle2}{\bullet}\!\!\!\!\!\ar@{-}[r]& \!\!\!\!\!\overset{\scriptscriptstyle3}{\circ}\!\!\!\!\!\ar@{-}[r]& \!\!\!\!\!\overset{\scriptscriptstyle4}{\circ}\!\!\!\!\!\ar@{-}[r]& \!\!\!\!\!\overset{\scriptscriptstyle5}{\bullet}\!\!\!\!\!\ar@{-}[r]& \!\!\overset{\scriptscriptstyle6}{\circ}}$\\[-3pt]
    $n\!=\!7$: & $\xymatrix@R=-1pt@C=5mm{\overset{\scriptscriptstyle0}{\circ}\!\!\ar@{-}[r]& \!\!\!\!\overset{\scriptscriptstyle1}{\circ}\!\!\!\!\ar@{-}[r]& \!\!\!\!\!\overset{\scriptscriptstyle2}{\bullet}\!\!\!\!\!\ar@{-}[r]& \!\!\!\!\!\overset{\scriptscriptstyle3}{\circ}\!\!\!\!\!\ar@{-}[r]& \!\!\!\!\!\overset{\scriptscriptstyle4}{\circ}\!\!\!\!\!\ar@{-}[r]& \!\!\!\!\!\overset{\scriptscriptstyle5}{\bullet}\!\!\!\!\!\ar@{-}[r]& \!\!\!\!\!\overset{\scriptscriptstyle6}{\circ}\!\!\!\!\!\ar@{-}[r]& \!\!\overset{\scriptscriptstyle7}{\circ}}$\\[-3pt]
    $n\!=\!8$: & $\xymatrix@R=-1pt@C=5mm{\overset{\scriptscriptstyle0}{\circ}\!\!\ar@{-}[r]& \!\!\!\!\overset{\scriptscriptstyle1}{\circ}\!\!\!\!\ar@{-}[r]& \!\!\!\!\!\overset{\scriptscriptstyle2}{\bullet}\!\!\!\!\!\ar@{-}[r]& \!\!\!\!\!\overset{\scriptscriptstyle3}{\circ}\!\!\!\!\!\ar@{-}[r]& \!\!\!\!\!\overset{\scriptscriptstyle4}{\circ}\!\!\!\!\!\ar@{-}[r]& \!\!\!\!\!\overset{\scriptscriptstyle5}{\bullet}\!\!\!\!\!\ar@{-}[r]& \!\!\!\!\!\overset{\scriptscriptstyle6}{\circ}\!\!\!\!\!\ar@{-}[r]& \!\!\!\!\!\overset{\scriptscriptstyle7}{\circ}\!\!\!\!\!\ar@{-}[r]& \!\!\overset{\scriptscriptstyle8}{\bullet}}$
\end{longtable}\vspace{-5pt}
\noindent Let $n\!=\!3k\!+\!r$ with $r\!\in\!\{0,1,2\}$. If $r\!=\!0$ and we add $\big\{\sigma\!\to\!\sigma\!\setminus\!\{n\}; n\!\in\!\sigma\big\}$ to $\mathcal{M}$, then $\mathring{\mathcal{M}}\!=\! \{\{n\}\}$. If $r\!\neq\!0$, then the critical simplices of $\Delta$ form the simplicial join $\bigast_{\!i=0}^k\!\S^0\approx\S^k$, so adding edges $\big\{\sigma\!\to\!\sigma\!\setminus\!\{m\!+\!1\};\, m\!=\!\min3\N\!\setminus\!\sigma,m\!+\!1\!\in\!\sigma\big\}$ to $\mathcal{M}$ gives $\smash{\mathring{\mathcal{M}}}= \{\{1\},\{0,3,6,\ldots,3k\}\}$. We conclude that $\Delta\simeq \smash{\icases{\mathrm{pt}}{n=3k}{\S^k }{n=3k+1,2}{\scriptstyle}{-1pt}{\Big}}$. \hfill$\lozenge$\\[5pt]
\end{Exp}

\vspace{4mm}
\section{Chevalley (Co)Homology}\label{section2}
In this section we recall the definitions and examples of Lie algebras and their (co)homology, from the viewpoint of digraphs and AMT.\vspace{2mm}

\subsection{Formulation}\label{2.formulation} \cite{citeWeibelIHA,citeLodayCH} Let $\frak{g}$ be a Lie $R$-algebra, i.e. an $R$-module with a binary operation $[-,-]\!:\frak{g}\!\times\!\frak{g}\!\longrightarrow\!\frak{g}$ that is $R$-linear in both variables and satisfies
$$\forall a,b,c\!\in\!\frak{g}\!:\, [a,a]\!=\!0,\, [[a,b],c]\!+\![[b,c],a]\!+\![[c,a],b]\!=\!0.$$
Let $\Lambda^{\!k}M=(\bigotimes_{i=1}^k\!M)/\langle m_1\!\otimes\!\ldots\!\otimes\!m_k; m_i\!=\!m_j\text{ for some }i\!\neq\!j\rangle$ be the exterior power of $R$-module $M$. The \emph{Chevalley (co)homology} of $\frak{g}$ with trivial coefficients, denoted $H_k(\frak{g};R)$ and $H^k(\frak{g};R)$, is the (co)homology of the chain complex of $R$-modules \vspace{-3pt}
\begin{longtable}[c]{c}
    $C_\ast\!:\;\;\; \ldots\longrightarrow \Lambda^{\!k}\!\frak{g} \overset{\partial_k}{\longrightarrow} \Lambda^{\!k\!-\!1}\!\frak{g} \longrightarrow \cdots \longrightarrow \Lambda^{\!2}\frak{g} \overset{[\;,\;]}{\longrightarrow} \frak{g} \overset{0}{\longrightarrow} R$,\\[0pt]
    $\partial_k(x_1\!\!\wedge\!\ldots\!\wedge\!x_k) = \sum_{r<s}(-1)^{r+s}[x_r,x_s]\!\wedge\!x_1\!\wedge\!\ldots\!\wedge\!\widehat{x_r}\!\wedge\!\ldots\!\wedge\!\widehat{x_s}\!\wedge\!\ldots\!\wedge\!x_k$.
\end{longtable}\vspace{-2pt}
\noindent Thus $H_0(\frak{g};R)\!\cong\! H^0(\frak{g};R)\!\cong\! R$. We shall deal only with Lie algebras that admit some $R$-module basis $\{e_i; i\!\in\!I\}\!\subseteq\!\frak{g}$, because then any linear ordering $\leq$ of $I$ implies that $\{e_{i_1}\!\wedge\!\ldots\!\wedge\!e_{i_k}; i_1\!<\!\ldots\!<\!i_k\}$ is a basis of $\Lambda^{\!k}\frak{g}$, so $\{e_{i_1}\!\wedge\!\ldots\!\wedge\!e_{i_k}; i_1\!<\!\ldots\!<\!i_k, k\!\in\!\N\}$ is the set of vertices of $\Gamma_{C_\ast}$. From now on, we will omit the $\wedge$ symbols to save space, and remember that $e_{\pi i_1}\!\ldots\!e_{\pi i_k}\!=\mathrm{sgn}\pi\,e_{i_1}\!\ldots\!e_{i_k}$ for any permutation $\pi$.
\par If we denote $H^\ast(\frak{g})\!=\!\bigoplus_{k\in\N}\!H^k(\frak{g})$, then diagonal map $\mathfrak{g}\longrightarrow\mathfrak{g}\!\times\!\mathfrak{g},x\!\mapsto\!(x,x)$ and zero map $\mathfrak{g}\!\longrightarrow\!0,x\!\mapsto\!0$ induce \emph{cup product} $\smile: H^\ast\!(\mathfrak{g})\!\otimes\!H^\ast\!(\mathfrak{g})\!=\! H^\ast(\mathfrak{g}\!\times\!\mathfrak{g})\longrightarrow H^\ast(\mathfrak{g})$ and $R\!=\!H^\ast(0)\!\longrightarrow\!H^\ast(\mathfrak{g})$ that make $H^\ast(\mathfrak{g})$ an associative unital graded-commutative $R$-algebra. Explicitly, for $\alpha\!\in\!H^i(\frak{g})$ and $\beta\!\in\!H^j(\frak{g})$ we have \vspace{-5pt}
$$(\alpha\!\smile\!\beta)(x_1\!\wedge\!\ldots\!\wedge\!x_{i+j})= \!\!\sum_{\substack{\pi\in S_{i+j},\,\pi_1<\ldots<\pi_i,\\ \pi_{i+1}<\ldots<\pi_{i+j}}}\!\!\! \mathrm{sgn}\pi\,\alpha(x_{\pi_1}\!\!\wedge\!\ldots\!\wedge\!x_{\pi_i})\, \beta(x_{\pi_{i+1}}\!\!\wedge\!\ldots\!\wedge\!x_{\pi_{i+j}}),\vspace{-5pt}$$
where the sum runs through all shuffle permutations.\vspace{2mm}

\subsection{Operations on complexes} If $R$ is a PID (such as $\Z$) and $\frak{g}$ as an $R$-module is free of finite rank $N$\!, then $\Lambda^{\!k}\frak{g}$ is free of rank $\binom{N}{k}$ and $H_k(\frak{g}),H^k(\frak{g})$ are finitely generated $R$-modules, subject to the classification theorem. Denote $F(-)\!=\,$free part, $T(-)\!=\,$torsion part, $T_p(-)\!=\,p$-torsion part. Universal coefficient theorem states: $F(H^k(\frak{g}))\!\cong\!F(H_k(\frak{g}))$, $T(H^k(\frak{g}))\!\cong\!T(H_{k\!-\!1}(\frak{g}))$; $H_k(\frak{g};\Z_p)\!\cong\! \big(F(H_k(\frak{g};\Z))\!\oplus\!T_p(H_k(\frak{g};\Z))\!\oplus\!T_p(H_{k\!-\!1}\!(\frak{g};\Z))\big)\!\otimes\!\Z_p$. K\"{u}nneth theorem states: if $R$ is a hereditary ring and $C_\ast$ and $D_\ast$ are chain complexes of $R$-modules, then
$$H_k(\mathrm{Hom}(C_\ast,D_\ast))\;\cong \prod_{\scriptscriptstyle i+j=k}\!\mathrm{Hom}\big(H_i(C_\ast),H_j(D_\ast)\big)\,\times\!\! \prod_{\scriptscriptstyle i+j=k\!-\!1}\!\!\!\mathrm{Ext}^1\!\big(H_i(C_\ast),H_j(D_\ast)\big)\vspace{-1mm}$$
when all $C_i$ are projective or all $D_j$ are injective, and\vspace{-1mm}
$$H_k(C_\ast\!\otimes\!D_\ast)\;\cong \bigoplus_{\scriptscriptstyle i+j=k}\!H_i(C_\ast)\!\otimes\!H_j(D_\ast)\,\oplus\!\!\bigoplus_{\scriptscriptstyle i+j=k\!-\!1}\!\!\!\mathrm{Tor}_1\!\big(H_i(C_\ast),H_j(D_\ast)\big)\vspace{-1mm}$$
when all $C_i$ or all $D_j$ are flat.\vspace{2mm}

\subsection{Examples} For any associative $R$-algebra $A$ we obtain a Lie $R$-algebra $L(A)$: as an $R$-module it is the same as $A$, but the bracket is given by $[a,b]\!=\!ab\!-\!ba$. Let $M_n(R)$ denote the associative unital $R$-algebra of all $n\!\times\!n$ matrices and $\frak{gl}_n(R)\!=\!L\big(M_n(R)\big)$. We have subalgebras $M_n^{\triangleu}(R)\!=\!\{a\!\in\!M_n(R); a_{ij}\!=\!0\text{ for }i\!\nleq\!j\}$ and $M_n^{\corneru}(R)\!=\!\{a\!\in\!M_n(R); a_{ij}\!=\!0\text{ for }i\!\nless\!j\}$ of all (strictly) upper triangular matrices, together with their Lie counterparts $\frak{sol}_n(R)\!=\!L\big(M_n^{\triangleu}(R)\big)$ and $\frak{nil}_n(R)\!=\!L\big(M_n^{\corneru}(R)\big)$. More generally, let $(P,\preceq)$ be any finite poset (or equivalently, an acyclic digraph) with elements $p_1,\ldots,p_n$. We obtain associative $R$-algebras and Lie $R$-algebras \vspace{-1pt}
\begin{longtable}[c]{l l l}
$M_n^\preceq(R)\!=\!\{a\!\in\!M_n(R); a_{ij}\!=\!0\text{ for }p_i\!\npreceq\!p_j\}$, & $\frak{gl}_n^\preceq(R)\!=\!L\big(M_n^{\preceq}(R)\big)$,\\[0pt]
$M_n^\prec  (R)\!=\!\{a\!\in\!M_n(R); a_{ij}\!=\!0\text{ for }p_i\!\nprec\!p_j\}$,   & $\frak{gl}_n^\prec(R)\!=\!L\big(M_n^{\prec}(R)\big)$.
\end{longtable}\vspace{-8pt}
\noindent If $\preceq$ is the usual linear ordering of $[n]$, then $\frak{gl}_n^{\preceq}(R)\!=\!\frak{sol}_n(R)$ and $\frak{gl}_n^{\prec}(R)\!=\!\frak{nil}_n(R)$. For theory it is convenient to assume $p_i\!=\!i$, but for interesting examples one needs abstract $p_i$ (such as when our poset is $(2^{[n]},\subseteq)$).\vspace{2mm}

\subsection{Motivation} Traditionally, the coefficients of most interest for (co)homology of a given Lie algebra are fields of characteristic $0$, because they calculate the singular (co)homology of the corresponding Lie group. Lie groups and their (co)homology are important in differential geometry/topology, because they are the objects that act on smooth manifolds, i.e. the appropriate representations of automorphism groups in the category of smooth manifolds. Smooth manifolds are analytic constructs, their algebraic analogue being \emph{algebraic sets} (solutions of a system of polynomial equations in several variables), which are defined over any field $K$, and the analogue of Lie groups are the \emph{algebraic groups} (algebraic sets with group structure, such that multiplication and inversion are polynomial maps). We expect a similar correspondence between algebraic groups over $K$ and Lie $K$\!-algebras; see \cite{citeTauvelYuLAAG}. Thus along with $H_k(\frak{g};\Q)$, it is also desirable to know about $H_k(\frak{g};\Z_p)$ for all $p$.

\vspace{4mm}
\section{Triangular matrices $\frak{sol}_n(R)$}\label{section3}
Our Lie algebra $\frak{sol}_n(R)$ is a subalgebra of $\frak{gl}_n(R)$ and is solvable. It admits an $R$-module basis $\{e_{ij}; 1\!\leq\!i\!\leq\!j\!\leq\!n\}$ where $e_{ij}$ is the matrix with $1$ at $(i,j)$-th entry and $0$ elsewhere. Programming with large matrices shows $H_k(\frak{sol}_n;\Z)$ for small $n$ below.
This is an infinite table of $\Z$-modules. In this section, $C_\ast$ denotes the chain complex for the homology of $\frak{g}\!=\!\frak{sol}_n$. Since $\frak{g}$ as an $R$-module is free of rank $N\!=\!\binom{n+1}{2}$, the matrix $\partial_k$ is of size $\binom{N}{k\!-\!1}\!\times\!\binom{N}{k}$ and $\Gamma_{C_\ast}$ has $2^N$ vertices. This exponential growth is problematic for computing the Smith normal form of a matrix: the largest matrices appear at the middle of the complex and are of width
{\tiny\begin{longtable}[c]{@{\hspace{0pt}}r@{\hspace{3pt}}|c@{\hspace{3pt}} c@{\hspace{3pt}} c@{\hspace{3pt}} c@{\hspace{3pt}} c@{\hspace{3pt}} c@{\hspace{3pt}} c@{\hspace{3pt}} c@{\hspace{3pt}} c@{\hspace{3pt}} c@{\hspace{3pt}} c@{\hspace{3pt}} c@{\hspace{3pt}} c@{\hspace{3pt}} c@{\hspace{3pt}} c@{\hspace{3pt}} c@{\hspace{3pt}} c@{\hspace{3pt}} c@{\hspace{3pt}} c@{\hspace{3pt}} c@{\hspace{3pt}}}
$n$& $1$& $2$& $3$& $4$& $5$& $6$& $7$& $8$& $9$& $10$& $11$& $12$& $13$& $14$& $15$& $16$& $17$& $18$& $19$& $20$\\
$\binom{N}{N/2}$& $10^1$& $10^1$& $10^2$& $10^3$& $10^4$& $10^6$& $10^8$& $10^{10}$& $10^{13}$& $10^{16}$& $10^{19}$& $10^{23}$& $10^{27}$& $10^{31}$& $10^{35}$& $10^{40}$& $10^{45}$& $10^{51}$& $10^{56}$& $10^{62}$.
\end{longtable}}
\noindent Dealing with such size is not feasible, even though the matrices are very sparse. For reference, the number of all atoms of the Earth is $\doteq\!10^{50}$. However, using the results from \ref{section3}, \ref{section4}, \ref{section7}, in an upcoming article \cite{citearticleLampretVavpeticTTLAS} we compute $H_\ast(\frak{sol}_n;\Z)$ for $n\!\leq\!8$.\vspace{3pt}
\begin{figure}[t] \small
\begin{longtable}[c]{@{\hspace{0pt}}r@{\hspace{5pt}}|l@{\hspace{5pt}} l@{\hspace{5pt}} l@{\hspace{5pt}} l@{\hspace{5pt}} l@{\hspace{5pt}} l@{\hspace{5pt}}}
$k \backslash n$ &$1$&$2$&$3$&$4$&$5$&$6$\\ \hline
$0$  &$\Z$ &$\Z$ &$\Z$ &$\Z$ &$\Z$ &$\Z$\\
$1$  &$\Z$ &$\Z^2$&$\Z^3$ &$\Z^4$ &$\Z^5$ &$\Z^6$\\
$2$  &&$\Z$ &$\Z^3$ &$\Z^6$ &$\Z^{10}$ &$\Z^{15}$\\
$3$  &&$0$ &$\Z\!\oplus\!\Z_2$ &$\Z^4\!\oplus\!\Z_2^3$ &$\Z^{10}\!\oplus\!\Z_2^6$ &$\Z^{20}\!\oplus\!\Z_2^{10}$\\
$4$  &&&$\Z_2^2$ &$\Z\!\oplus\!\Z_2^{11}$ &$\Z^5\!\oplus\!\Z_2^{29}$ &$\Z^{15}\!\oplus\!\Z_2^{59}$\\
$5$  &&&$\Z_2$ &$\Z_2^{15}\!\oplus\!\Z_3$ &$\Z\!\oplus\!\Z_2^{56}\!\oplus\!\Z_3^3$ &$\Z^{6}\!\oplus\!\Z_2^{145}\!\!\oplus\!\Z_3^6$\\
$6$  &&&$0$ &$\Z_2^9\!\oplus\!\Z_3^3$ &$\Z_2^{59}\!\oplus\!\Z_3^{13}$ &$\Z\!\oplus\!\Z_2^{220}\!\!\oplus\!\Z_3^{33}$\\
$7$  &&&&$\Z_2^2\!\oplus\!\Z_3^3$ &$\Z_2^{51}\!\oplus\!\Z_4\!\oplus\!\Z_3^{22}$ &$\Z_2^{348}\!\!\oplus\!\Z_4^3\!\oplus\!\Z_3^{75}$\\
$8$  &&&&$\Z_3$ &$\Z_2^{55}\!\oplus\!\Z_4^4\!\oplus\!\Z_3^{19}$ &$\Z_2^{674}\!\!\oplus\!\Z_4^{16}\!\oplus\!\Z_3^{96}$\\
$9$  &&&&$0$ &$\Z_2^{50}\!\oplus\!\Z_4^6\!\oplus\!\Z_3^{11}$ &$\Z_2^{1034}\!\!\oplus\!\Z_4^{35}\!\oplus\!\Z_3^{94}\!\oplus\!\Z_5^1$\\
$10$ &&&&$0$ &$\Z_2^{26}\!\oplus\!\Z_4^4\!\oplus\!\Z_3^7$ &$\Z_2^{1035}\!\!\oplus\!\Z_4^{40}\!\oplus\!\Z_3^{101}\!\!\oplus\!\Z_5^5$\\
$11$ &&&&&$\Z_2^9\!\oplus\!\Z_4\!\oplus\!\Z_3^4$ &$\Z_2^{704}\!\!\oplus\!\Z_4^{25}\!\oplus\!\Z_3^{103}\!\!\oplus\!\Z_5^{10}$\\
$12$ &&&&&$\Z_2^6\!\oplus\!\Z_3$ &$\Z_2^{452}\!\!\oplus\!\Z_4^{9}\!\oplus\!\Z_3^{70}\!\oplus\!\Z_5^{10}$\\
$13$ &&&&&$\Z_2^4$ &$\Z_2^{389}\!\!\oplus\!\Z_4^{6}\!\oplus\!\Z_3^{26}\!\oplus\!\Z_5^5$\\
$14$ &&&&&$\Z_2$ &$\Z_2^{305}\!\!\oplus\!\Z_4^{10}\!\oplus\!\Z_3^4\!\oplus\!\Z_5$\\
$15$ &&&&&$0$ &$\Z_2^{150}\!\!\oplus\!\Z_4^{10}$\\
$16$ &&&&&&$\Z_2^{39}\!\oplus\!\Z_4^5$\\
$17$ &&&&&&$\Z_2^4\!\oplus\!\Z_4$\\
$18$ &&&&&&$0$\\
$19$ &&&&&&$0$\\
$20$ &&&&&&$0$\\
$21$ &&&&&&$0$
\end{longtable}\vspace{-3mm}
\label{3.table}
\end{figure}

\subsection{Digraph} Every edge of $\Gamma_{C_\ast}$ is of the form
$\begin{smallmatrix}\ldots e_{aa}e_{ab}\ldots\\[-1pt] \downarrow\\[-2pt] \ldots e_{ab}\ldots\end{smallmatrix}$ or
$\begin{smallmatrix}\ldots e_{ab}e_{bb}\ldots\\[-1pt] \downarrow\\[-2pt] \ldots e_{ab}\ldots\end{smallmatrix}$ or
$\begin{smallmatrix}\ldots e_{ax}e_{xb}\ldots\\[-1pt] \downarrow\\[-2pt] \ldots e_{ab}\ldots\end{smallmatrix}$ for some $a\!<\!x\!<\!b$, because $[e_{ab},e_{cd}]\!=\! \delta_{bc}e_{ad} \!-\! \delta_{ad}e_{cb}$. The last edge has weight $\pm1$, but this may not hold for the first two. For any $\sigma\!=\!\{x_1,\ldots,x_k\}\!\subseteq\![n]$ we denote $e_\sigma\!=\!e_{x_1x_1}\!\ldots\!e_{x_kx_k}$, the wedge of the corresponding diagonals.\vspace{1mm}

\begin{Prp} $\{\text{isolated vertices of }\Gamma_{C_\ast}\}\!=\!\{e_\sigma; \sigma\!\subseteq\![n]\}$ when $\Z\!\leq\!R$.
\end{Prp}\vspace{-1mm}
\noindent Thus $R^{\binom{n}{k}}$ is a direct summand of $H_k(\frak{sol}_n)$ and $H^k(\frak{sol}_n)$.
\begin{proof}
The Lie bracket of any two diagonals $e_{aa}$ and $e_{bb}$ is $0$. The only splitting of a diagonal is $e_{aa}\!=\!e_{ax}\!\cdot\!e_{xa}$, but then either $e_{ax}\!\notin\!\frak{sol}_n$ or $e_{xa}\!\notin\!\frak{sol}_n$. Thus every $e_{\sigma}$ is isolated. If a vertex $v$ contains a nondiagonal element $e_{ab}$, then $v$ either contains $e_{aa}$ (hence omitting $e_{aa}$ makes $v$ initial) or it doesn't (hence adding $e_{aa}$ makes $v$ terminal). Thus every isolated vertex is of the form $e_\sigma$. Here is the tricky part: adding or removing a diagonal can be an edge with weight $0$ (so not really an edge). However, for every vertex $v$ that contains a nondiagonal $e_{ab}$ there exists a diagonal $e_{xx}$ with nonzero weight (see the proof of \ref{3.matching} and \ref{3.Q} below).
\end{proof}\vspace{1mm}

Over any ring $R$, there holds row growth:\vspace{-1mm}
\begin{Prp}$H_k(\frak{sol}_n)\!\cong\! H_k(\frak{sol}_{n\!-\!1})\!\oplus\!\ldots$ and $H^k(\frak{sol}_n)\!\cong\! H^k(\frak{sol}_{n\!-\!1})\!\times\!\ldots$.
\end{Prp}
\begin{proof}
The largest and the smallest index that appears in a vertex $v$ appears also in any vertex adjacent to $v$. Hence chain complex $C_\ast(\frak{sol}_n)$ is a direct sum of $C_\ast(\frak{sol}_{n\!-\!1})$ and the subcomplex spanned by all vertices that contain index $n$.
\end{proof}\vspace{1mm}

\begin{Prp} $H_{2n\!-\!3}(\frak{sol}_n;\Z) \cong \Z_{n\!-\!1}\!\oplus\!\ldots$.
\end{Prp}\vspace{-1mm}
\noindent Thus every prime power $\Z_{p^l}$ appears in the table as a direct summand.
\begin{proof}
The vertex $v\!=\!e_{12}e_{13}e_{14}\!\ldots\!e_{1n}e_{2n}e_{3n}\!\ldots\!e_{n\!-\!1,n}$ is initial to no edge (hence $v\!\in\!\mathrm{Ker}\,\partial_{2n-3}$) and terminal to two edges, namely to $e_{\{1\}}v\!\to\!v$ and $e_{\{n\}}v\!\to\!v$, whose weights are $\pm(n\!-\!1)$ (hence $(n\!-\!1)v\!\in\!\mathrm{Im}\,\partial_{2n-2}$).
\end{proof}\vspace{1mm}

\begin{Cnj}\label{3.CnjPrimePower} For every prime power $m\!=\!p^l$ the module $\Z_{m}$ appears in $H_\ast(\frak{sol}_\ast)$ for the first time in $m\!+\!1$-th column and $2m\!-\!1$-th row.
\end{Cnj}\vspace{-1mm}
The case $l\!=\!1$ (over a field) is proved in \ref{3.Q} and \ref{3.Zp}.\vspace{2mm}

\subsection{Matching}\label{3.2} The increasing wedges of $e_{ab}$'s constitute the bases for modules $C_k(\frak{g})$. By permuting elements of a wedge, we obtain the same wedge together with the sign of the permutation. Such permuted wedges constitute new bases for $C_k(\frak{g})$. Every vertex of digraph $\Gamma_{C_\ast}$ can be written in the form $v\!=\!\pm e_{a_1b_1}\!\ldots\!e_{a_lb_l}e_\sigma$ with all $a_i\!<\!b_i$, and then we denote $\epsilon(v)\!=\! \{a_1,b_1,\ldots,a_l,b_l\}$. For any index $x\!\in\!\epsilon(v)$ we can write vertex $v$ in the form $\pm e_{a_1x}\!\ldots\!e_{a_rx}\overline{e_{xx}}e_{xb_1}\!\ldots\!e_{xb_s} \!e_{c_1d_1}\!\ldots\!e_{c_td_t}e_\sigma$ with $x\!\notin\!\{c_1,d_1,\ldots,c_t,d_t\}$ and $x\!\notin\!\sigma$, where overline means $e_{xx}$ is either missing or present in the wedge, so omitting $e_{xx}$ is an edge with weight $(-1)^r(r\!-\!s)$. Here $r\!=\!r_x$ is the number of times $x$ appears on the right and $s\!=\!s_x$ is the number of times $x$ appears on the left. We define $$\mathcal{M}=\Big\{
\begin{smallmatrix}e_{a_1x}\ldots e_{a_rx}e_{xx}          e_{xb_1}\!\ldots e_{xb_s} e_{c_1d_1}\!\ldots e_{c_td_t}e_\sigma\\[-2pt] \hspace{-48pt}\downarrow\\[-1pt]
                   e_{a_1x}\ldots e_{a_rx}\widehat{e}_{xx}e_{xb_1}\!\ldots e_{xb_s} e_{c_1d_1}\!\ldots e_{c_td_t}e_\sigma\end{smallmatrix}\!;\;
\begin{smallmatrix}x\in\epsilon(v)\text{ is minimal such }\\\text{that }r-s\text{ is a unit of }R \end{smallmatrix}\!\Big\}.$$

\begin{Prp}\label{3.matching} $\mathcal{M}$ is a Morse matching, with $\mathring{\partial}\!=\!\partial$ and $$\mathring{\mathcal{M}}\!=\!\big\{v;\, \forall x\!\in\!\epsilon(v)\!:r_x\!\!-\!s_x\!\in\!R\!\setminus\!R^\times\big\}.$$
\end{Prp}
\noindent Thus only the invertibility of integer multiples of unity $1_{\!R}$ determine $\mathcal{M}$.
\begin{proof}
Let us make sure that the edge above truly has weight $(-1)^r(r\!-\!s)$. Obtaining
$$\begin{smallmatrix}e_{a_1x}\ldots e_{a_ix}\ldots e_{a_rx}e_{xx}          e_{xb_1}\!\ldots e_{xb_j}\!\ldots e_{xb_s} e_{c_1d_1}\!\ldots e_{c_td_t}e_\sigma\\[-2pt]
                    \hspace{-45pt}\downarrow\\[-2pt]
                     e_{a_1x}\ldots e_{a_ix}\ldots e_{a_rx}\widehat{e}_{xx}e_{xb_1}\!\ldots e_{xb_j}\!\ldots e_{xb_s} e_{c_1d_1}\!\ldots e_{c_td_t}e_\sigma\end{smallmatrix}$$
is only possible when we take the bracket $[e_{a_ix},e_{xx}]\!=\!e_{a_ix}$ or $[e_{xx},e_{xb_j}]\!=\!e_{xb_j}$ for all $i$ and $j$, because for all other nonzero brackets the element $e_{xx}$ remains in the wedge. The first bracket has sign $(-1)^{i+(r+1)+(i-1)}\!=\!(-1)^r$ and there are $r$ choices for $i$. The second bracket has sign $(-1)^{(r+1)+(r+1+j)+(r+j-1)}\!=\!-(-1)^r$ and there are $s$ choices for $j$. Together this gives sign $(-1)^r(r\!-\!s)$ and by definition of $\mathcal{M}$ this is a unit of $R$, so condition $(2)$ from \ref{1.1} is satisfied.
\par For every edge in $\mathcal{M}$, the initial and terminal vertex have the same set $\epsilon(-)$, so $\overset{\scriptscriptstyle\mathcal{M}}{\longrightarrow} \overset{\scriptscriptstyle\mathcal{M}}{\longrightarrow}$ does not occur in $\Gamma_{\!C_\ast}$, because we cannot twice remove $e_{xx}$. Since $x$ is uniquely determined by the set $\epsilon(v)$, situations $\overset{\scriptscriptstyle\mathcal{M}}{\longleftarrow} \overset{\scriptscriptstyle\mathcal{M}}{\longrightarrow}$ and $\overset{\scriptscriptstyle\mathcal{M}}{\longrightarrow} \overset{\scriptscriptstyle\mathcal{M}}{\longleftarrow}$ do not occur in $\Gamma_{\!C_\ast}$. This shows that condition $(1)$ from \ref{1.1} is satisfied.
\par Since digraph $\Gamma_{C_\ast}$ is finite, there cannot exist an infinite zig-zag path in $\Gamma_{C_\ast}^{\mathcal{M}}$. For $\begin{smallmatrix}u=\ldots e_{ax}e_{xb}\ldots\\[-1pt] \hspace{10pt}\downarrow\\[-2pt] v=\ldots e_{ab}\ldots\end{smallmatrix}$ and any $i\!\in\!\epsilon(u)$, number $r_i\!-\!s_i$ calculated in $u$ and in $v$ is the same. Given $u\!\overset{\scriptscriptstyle\mathcal{M}}{\longrightarrow}\!v$ which is the omission of $e_{xx}$, for any edge $u'\!\longrightarrow\!v$, going from $v$ to $u'$ is either adding some $e_{yy}\!\notin\!v$ or splitting some $e_{ab}\!\in\!v$. In both cases, the minimal $y\!\in\!\epsilon(u')$ with $r_y\!-\!s_y\!\in\!R^\times$ is still $x$, but since $e_{xx}\!\notin\!u'$, the edge $u'\!\overset{\scriptscriptstyle\mathcal{M}}{\longrightarrow}\!v$ is not possible and $u'$ is terminal (hence not critical). Thus there are no directed cycles in $\Gamma_{C_\ast}^{\mathcal{M}}$, affirming condition $(3)$ from \ref{1.1}, and every zig-zag path between critical vertices contains no edge from $\mathcal{M}$, so $\mathring{\partial}$ is the restriction of $\partial$ to $\mathring{\mathcal{M}}$.
\end{proof}\vspace{1mm}

\begin{Exp} Below is a subgraph corresponding to a direct summand of the chain complex for $\frak{sol}_4$. The edges are labeled by their weight.\\
\includegraphics[width=1.0\textwidth]{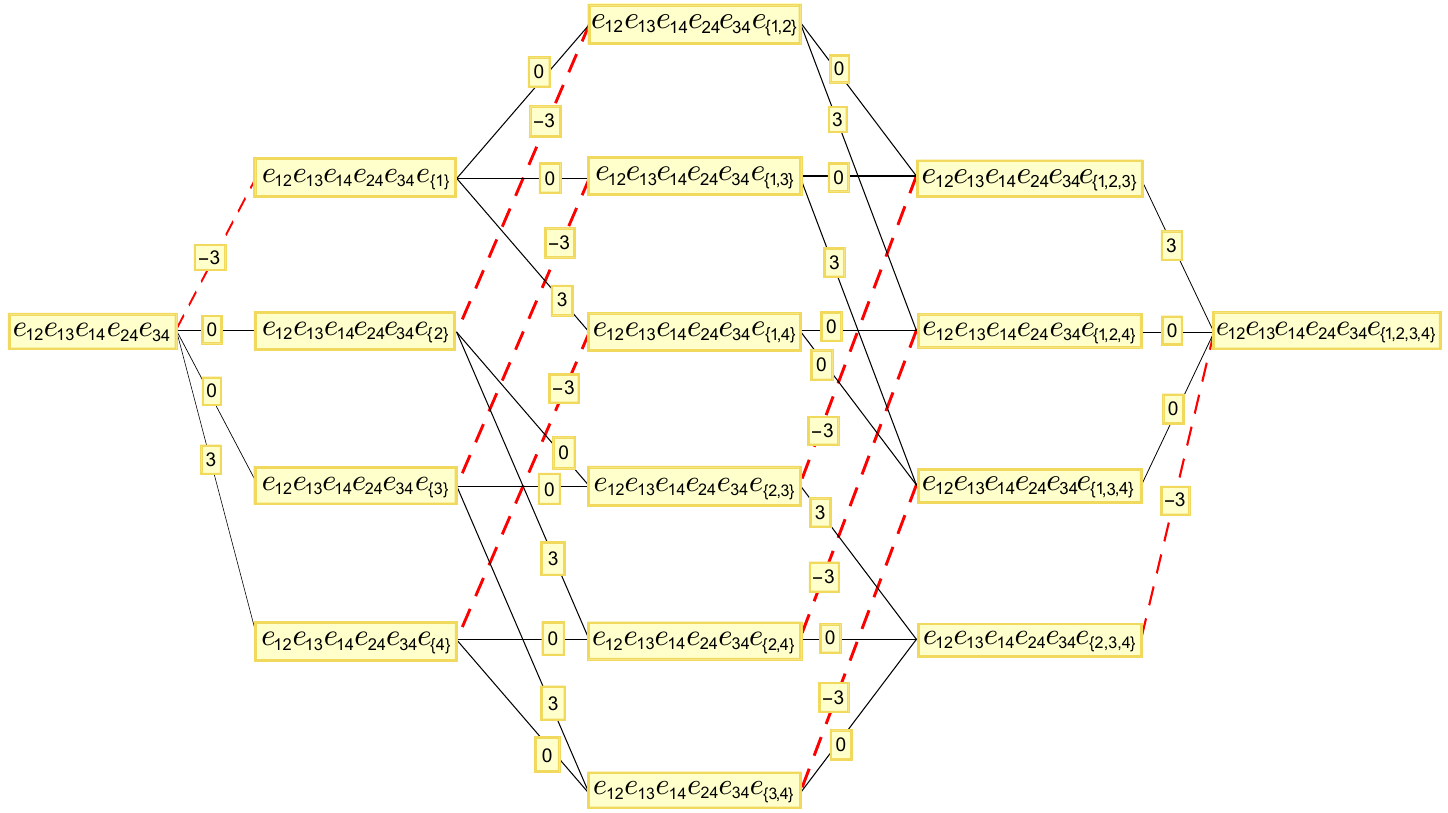}\\
If $\Z_3\!\leq\!R$, then all weights are zero, so there are no edges and every vertex is a generator in (co)homology. But if $\Q\!\leq\!R$ or $\Z_p\!\leq\!R$ where $p\!\neq\!3$, then the weights $\pm3$ are invertible, so the matching $\mathcal{M}$ defined above is marked by dashed red lines: there are no critical vertices, hence (co)homology is trivial. \hfill$\lozenge$
\end{Exp}\vspace{2mm}

\subsection{Low Degrees} We are interested in the first three rows of the table.

\begin{Prp}\label{3.LowDegrees} If $\Z\!\cap\!R^\times\!=\!\{\pm1\}$, such as in the case $R\!=\!\Z$, then there holds $H_1(\frak{sol}_n)\!\cong\!H^1\!(\frak{sol}_n) \cong R^n$\!, $H_2(\frak{sol}_n)\!\cong\!H^2\!(\frak{sol}_n) \cong R^{\binom{n}{2}}$\!, $H_3(\frak{sol}_n) \!\cong\! R^{\binom{n}{3}}\!\oplus\!(\frac{\scriptscriptstyle R}{\scriptscriptstyle 2R})^{\binom{n\!-\!1}{2}}$\!.
\end{Prp}
\begin{proof} A nice combinatorial exercise finds the elements of $\mathring{\mathcal{M}}$ and their boundary:\vspace{-2pt}
\begin{longtable}[c]{ll}
$\mathring{\mathcal{M}}_1$:  &$\mathring{\partial}e_{\{a\}}=0$;\hspace{40mm}$w_i\!=\!r_i\!-\!s_i$\\[-1pt]
$\mathring{\mathcal{M}}_2$:  &$\mathring{\partial}e_{\{a,b\}}=0$;\\[-1pt]
$\mathring{\mathcal{M}}_3$:  &$\mathring{\partial}e_{\{a,b,c\}}=0, \mathring{\partial}e_{ab}e_{ac}e_{bc}=0$, $w_i\!=\!-2,0,2$;\\[-1pt]
$\mathring{\mathcal{M}}_4$:  &$\mathring{\partial}e_{\{a,b,c,d\}}=0$,
$\mathring{\partial}e_{\{x\}}\!e_{ab}e_{ac}e_{bc}=\icases{\pm2e_{ab}e_{ac}e_{bc}}{\text{if }x\in\{a,c\}}{0}{\text{otherwise}}{\scriptstyle}{2pt}{\Big}$,\\[-1pt]
&$\mathring{\partial}e_{ax}e_{ay}e_{xc}e_{yc}=-e_{ay}e_{ac}e_{yc}\!-\!e_{ax}e_{ac}e_{xc}$,\; $w_i\!=\!-2,0,0,2$,\\[-1pt]
&$\mathring{\partial}e_{ax}e_{xy}e_{yc}e_{ac}=-e_{ay}e_{yc}e_{ac}\!+\!e_{ax}e_{xc}e_{ac}$,\; $w_i\!=\!-2,0,0,2$,\\[-1pt]
&$\mathring{\partial}e_{ay}e_{ac}e_{xy}e_{xc}=0$,\; $w_i\!=\!-2,-2,2,2$.
\end{longtable}\vspace{-6pt}
\noindent Therefore $H_3(\frak{g})\!\cong\!\frac{\Ker\mathring{\partial}_3}{\Im\mathring{\partial}_4} \!\cong\! \smash{\frac{\langle e_{\{a,b,c\}},e_{ab}e_{ac}e_{bc}\rangle}{\langle e_{ax}e_{ac}e_{xc}-e_{ay}e_{ac}e_{yc}, 2e_{ab}e_{ac}e_{bc}\rangle}}$. Now $e_{\{a,b,c\}}$ is untouched by the relations, so it gives a summand $R^{\binom{n}{3}}$. For the rest, relation $2e_{ab}e_{ac}e_{bc}$ implies we have only $\frac{\scriptscriptstyle R}{\scriptscriptstyle 2R}$ summands, and relation $e_{ax}e_{ac}e_{xc}\!=\!e_{ay}e_{ac}e_{yc}$ determines how many: since $x$ is interchangeable with any $y$, it might as well be just $a\!+\!1$, so as many as the cardinality of the set $\{(a,c); 1\!\leq\!a\!<a\!+\!1\!<\!c\!\leq\!n\}$, which is $\binom{n\!-\!1}{2}$.
\end{proof}\vspace{2mm}

\subsection{High Degrees} We study the strange patterns of zeros at the bottom of \ref{3.table}. We wish to know when the first (from the bottom up) nonzero module appears, and over which $\Z_p$. Let $N\!=\!\binom{n+1}{2}$ be the dimension of $\frak{g}\!=\!\frak{sol}_n$. \vspace{1mm}

\begin{Prp}\label{3.HighDegrees} If $n\!=\!2m$, then $H_{\!N\!-k\!}(\frak{g};\Z_2)\cong0\ncong H_{\!N\!-m\!}(\frak{g};\Z_2)$ for all $k\!<\!m$. If $n\!=\!2m\!-\!1$, then $H_{\!N\!}(\frak{g};\Z_2)\cong\Z_2$. If $n\!=\!3m$, then $H_{\!N\!-k\!}(\frak{g};\Z_3) \cong0$ for all $k\!\leq\!m$. If $n\!=\!3m\!\pm\!1$, then $H_{\!N\!-k\!}(\frak{g};\Z_3) \cong0\ncong\Z_3\cong H_{\!N\!-m\!}(\frak{g};\Z_3)$ for all $k\!<\!m$. If $p\!\geq\!5$ is prime, then $H_{\!N\!-k\!}(\frak{g};\Z_p) \cong0$ for $k<\!\lfloor\frac{n}{p}\rfloor\frac{p^2\!-1}{8}$ or $k\leq\frac{n}{2}$.
\end{Prp}
Thus the top dimensional homology consists of $2$-torsion for $n\!\notin\!2\N$, it consists of $3$-torsion for $2\N\!\ni\!n\!\notin\!3\N$, and $2$-torsion or $3$-torsion for $n\!\in\!6\N$.
\begin{proof} 
Denote $e\!=\!\bigwedge_{a\leq b}\!e_{ab}$, the wedge of all elements, and notice that its weights are $w_i\!=\!r_i\!-\!s_i\!=\! (i\!-\!1)\!-\!(n\!-\!i) \!=\!2i\!-\!n\!-\!1 \!=\!1\!-\!n,3\!-\!n,5\!-\!n,\ldots,n\!-\!5,n\!-\!3,n\!-\!1$ for $i\!=\!1,\ldots,n$. Thus $C_N\!=\!\big\langle e\big\rangle$, $C_{N\!-\!1}\!=\!\big\langle e\!\setminus\!e_{aa},e\!\setminus\!e_{ab}\big\rangle$, $C_{N\!-\!2}\!=\!\big\langle e\!\setminus\!e_{aa}e_{bb},e\!\setminus\!e_{aa}e_{bc},e\!\setminus\!e_{ab}e_{cd}\big\rangle$, etc.
Removing $e_{ij}$ from $e$ decreases $s_i$ and $r_j$, hence increases $w_i$ and decreases $w_j$ by $1$.
\par Over $\Z_2$ there are two cases in which the weights appear:\vspace{-6pt}
\begin{longtable}[c]{ll}
$n\!=\!2m\!-\!1$: & $w_i\equiv0,0,\ldots,0$,\\[-1pt]
$n\!=\!2m$:       & $w_i\equiv1,1,\ldots,1,1$.
\end{longtable}\vspace{-8pt}
\noindent In the first case, $\mathring{\mathcal{M}}_N\!=\!\{e\}$ and $\mathring{\partial}_N\!=\!0$, so $H_{\!N\!}(\frak{g};\Z_2)\cong\Z_2$. In the second case, we must remove at least $m$ nondiagonals to obtain a critical vertex, so $H_{\!N\!-k\!}(\frak{g};\Z_2)\cong0$ for $k\!<\!m$. Since $\mathring{\mathcal{M}}_{\!N\!-m\!}$ contains $e\!\setminus\!e_{12}e_{34}\!\cdots e_{n\!-\!1,n} \!\in\!\mathrm{Ker}\,\mathring{\partial}$, we get $H_{\!N\!-m\!}(\frak{g};\Z_2)\ncong0$. In fact, $\mathring{\mathcal{M}}_{\!N\!-m\!}=\big\{e\!\setminus\!e_{a_1\!b_1}\!\cdots e_{a_m\!b_m}\!;\, \{a_1,b_1,\ldots,a_m,b_m\}\!=\![2m]\big\}$.
\par Over $\Z_3$ there are three cases in which the weights appear:\vspace{-6pt}
\begin{longtable}[c]{ll}
$n\!=\!3m\!-\!1$: & $w_i\equiv-1,1,0,\,-1,1,0,\,-1,1,0,\,\ldots,\,-1,1$,\\[-1pt]
$n\!=\!3m$:       & $w_i\equiv1,0,-1,\,1,0,-1,\,1,0,-1,\,\ldots,\,1,0,-1$,\\[-1pt]
$n\!=\!3m\!+\!1$: & $w_i\equiv0,-1,1,\,0,-1,1,\,0,-1,1,\,\ldots,\,0,-1,1,\,0$.
\end{longtable}\vspace{-8pt}
\noindent We must remove at least $m$ nondiagonals to obtain $\mathring{\mathcal{M}}_{\!N\!-k\!} \!\neq\!\{\}$, so $H_{\!N\!-k\!}(\frak{g};\Z_3)\cong0$ for $k\!<\!m$. In the first case, $\mathring{\mathcal{M}}_{\!N\!-m\!} \!=\!\{e\!\setminus\!e_{12}e_{45}e_{78}\!\cdots e_{n\!-\!1,n}\}$, hence $H_{\!N\!-m\!}(\frak{g};\Z_3)\cong\Z_3$. In the third case, $\mathring{\mathcal{M}}_{\!N\!-m\!} \!=\!\{e\!\setminus\!e_{23}e_{56}e_{89}\!\cdots e_{n\!-\!2,n\!-\!1}\}$, hence $H_{\!N\!-m\!}(\frak{g};\Z_3)\cong\Z_3$.
\par Over $\Z_p$ for prime $p\!\geq\!5$, there are $p$ cases in which the weights appear. Already for $p\!=\!5$ it is difficult to determine the minimal number of removals:\vspace{-4pt}
\begin{longtable}[c]{ll}
$n\!=\!5m$:         & $w_i\equiv 1,3,0,2,4,\,1,3,0,2,4,\,\ldots,\,1,3,0,2,4$,\\[-1pt]
$n\!=\!5m\!+\!1$:   & $w_i\equiv 0,2,4,1,3,\,0,2,4,1,3,\,\ldots,\,0,2,4,1,3,\,0$,\\[-1pt]
$n\!=\!5m\!+\!2$:   & $w_i\equiv 4,1,3,0,2,\,4,1,3,0,2,\,\ldots,\,4,1,3,0,2,\,4,1$,\\[-1pt]
$n\!=\!5m\!+\!3$:   & $w_i\equiv 3,0,2,4,1,\,3,0,2,4,1,\,\ldots,\,3,0,2,4,1,\,3,0,2$,\\[-1pt]
$n\!=\!5m\!+\!4$:   & $w_i\equiv 2,4,1,3,0,\,2,4,1,3,0,\,\ldots,\,2,4,1,3,0,\,2,4,1,3$.
\end{longtable}\vspace{-8pt}
\noindent The first $p$ weights are the permuted elements of $\Z_p$. To get a critical vertex, we must remove at least $\frac{1}{2}\sum_{i=1}^n\!\min\!\big(w_i\%p,\:p\!-\!(w_i\%p)\big)$ elements from $e$, where $\%$ is the remainder, because every $w_i$ must either be increased to $p$ or decreased to $0$. This sum has lower bounds $\lfloor\frac{n}{p}\rfloor\frac{p^2\!-1}{8}$ (ideally $1,2,\ldots,\frac{p-1}{2}$ is decreased and $\frac{p+1}{2},\ldots,p\!-\!1$ is increased, which requires \smash{$\sum_{i=1}^{(p-1)/2}\!i=\frac{p^2-1}{8}$} removals) and $\frac{n}{2}$ (for any zero weight we have $|\Z_p\!\setminus\!\{0,\pm1\}| \!=\!p\!-\!3$ weights that require several removals).
\end{proof} \vspace{1mm}
Fields $\Z_2$ and $\Z_3$ are different than other $\Z_p$, since only in the former do we have just elements $0,\pm1$ (which require one removal per two weights). This explains why in the table, summands $\Z_p$ with large $p$ appear at the middle of columns.
\par Case $n\!=\!3m$ is complicated: to obtain a critical vertex we must remove $1,n\!\in\!\epsilon(e)$ twice (i.e. $v\!=\!e\!\setminus\!e_{1x}e_{1y}e_{an}e_{bn}\cdots$), but removing $e_{ij}$ with $j\!-\!i\!\geq\!2$ means $v\notin\Ker\,\mathring{\partial}$. E.g. for $n\!=\!6$ we have $\mathring{\mathcal{M}}_N\!=\! \mathring{\mathcal{M}}_{N\!-1}\!=\! \mathring{\mathcal{M}}_{N\!-2}\!=\!\{\}, \mathring{\mathcal{M}}_{N\!-3}\!=\!\{e\!\setminus\!e_{14}e_{16}e_{36}\}, \mathring{\mathcal{M}}_{N\!-4}\!=\!\{e\!\setminus\!e_{15}e_{16}e_{34}e_{56}, e\!\setminus\!e_{14}e_{16}e_{35}e_{56}, e\!\setminus\!e_{14}e_{16}e_{34}e_{46}, e\!\setminus\!e_{14}e_{15}e_{36}e_{56}, e\!\setminus\!e_{13}e_{16}e_{34}e_{36}, e\!\setminus\!e_{12}e_{16}e_{26}e_{34}, e\!\setminus\!e_{12}e_{16}e_{24}e_{36}, e\!\setminus\!e_{12}e_{14}e_{26}e_{36}\}$, but $H_{N\!-k}(\frak{g};\Z_3)\cong0$ for $k\!=\!0,\ldots,5$.\vspace{2mm}

\subsection{Invertible Integers} Next we compute (co)homology of $\frak{sol}_n$ over $\R$ or $\C$.
\begin{Prp}\label{3.Q} If $\Q\!\leq\!R$ or $\Z_p\!\leq\!R$ with prime $p\!\geq\!n$ or $p\!>\!\frac{k+1}{2}$, then\vspace{-1mm} $$H_k(\frak{sol}_n)\cong H^k(\frak{sol}_n)\cong R^{\binom{n}{k}}.$$
\end{Prp}\vspace{-1mm}
Thus the free part of $H_k(\frak{sol}_n;\Z)$ and $H^k(\frak{sol}_n;\Z)$ is $\Z^{\binom{n}{k}}$.
\begin{proof} Every nonzero element of $\Q$ and $\Z_p$ is a unit, so over $\Z_p$ critical vertices are $\mathring{\mathcal{M}}\!=\!\{v;\, \forall x\!\in\!\epsilon(v)\!:r_x\!-\!s_x\!\in\!p\Z\}$. For any index $x\!\in\!\epsilon(v)$ we have $|r_x\!-\!s_x|\!<\!n$, because only $1$ and $n$ can appear $n\!-\!1$ times as an index, all other indices appear fewer times. For any vertex $e_\sigma\!\neq\!v\!\in\!C_k$ there exists $x\!\in\!\epsilon(v)$ with $0\!<\!|r_x\!-\!s_x|\!\leq\!\frac{k+1}{2}$: if this is false, then $a\!:=\!\min\epsilon(v)$ and $b\!:=\!\max\epsilon(v)$ appear at least $\frac{k+1}{2}\!=:\!m$ times each (only once can they appear together), so $m\!+\!m\!-\!1\!=\!k$ implies $v\!=\!e_{a_1b}\!\ldots\!e_{a_{m\!-\!1}b}e_{ab}e_{ab_1}\!\ldots\!e_{ab_{m\!-\!1}}$ and thus $r_{a_1}\!\!-\!s_{a_1}\!=\!-1$, a contradiction. Thus under our assumptions, the only noninvertible weight $r_x\!-\!s_x$ is $0\!\in\!\Z$, so $\mathring{\mathcal{M}}\!=\! \big\{v;\, \forall x\!\in\!\epsilon(v)\!: r_x\!=\!s_x\big\} \!=\!\{e_\sigma\}$. The last equality holds, because if $\epsilon(v)$ is nonempty, then it contains a maximal element $x$ (which can only appear on the right), but then $r_x\!>\!0$ and $s_x\!=\!0$.
\end{proof}\vspace{2mm}

\subsection{Invertible Integers mod $p$} Now we are interested in the first occurrences of primes $p$ that appear in the table as $\Z_{p^l}$ for some $l$. Using \ref{3.Q}, we obtain:

\begin{Prp}\label{3.Zp} Let $\Z_p\!\leq\!R$ for some prime $p$.
If $n\!=\!p\!+\!1$, then $H_k(\frak{sol}_n) \cong H^k\!(\frak{sol}_n) \cong R^{\binom{n}{k}}\!\oplus\! R^{\binom{n}{k\!-\!2n\!+\!3}}$.
If $k\!=\!2p\!-\!1$, then $H_k(\frak{sol}_n) \cong  R^{\binom{n}{k}}\!\oplus\! R^{\binom{n\!-\!p+1}{2}}$.
\end{Prp}
\noindent Thus in the table, the first column where $p$-torsion appears is the $p\!+\!1$-th column, and the first row in which $p$-torsion appears is the $2p\!-\!1$-th row. We also know in what quantity the homology of that first column/row appears. See also \ref{3.CnjPrimePower}.
\begin{proof} If $n\!=\!p\!+\!1$, then in a critical vertex $v$ every weight is a multiple of $p$, which is so large that necessarily $\min\epsilon(v)\!=\!1$ and $\max\epsilon(v)\!=\!n$, hence we have $\mathring{\mathcal{M}}\!=\! \{e_\sigma,e_\sigma e_{12}e_{13}e_{14}\!\ldots\!e_{1n}e_{2n}e_{3n}\!\ldots\!e_{n\!-\!1,n}\}$ and $\mathring{\partial}\!=\!0$, so the result is apparent.
\par If $k\!=\!2p\!-\!1$, then in a critical vertex $v$ the index $\min\epsilon(v)$ appears only on the left ($p$ times) and $\max\epsilon(v)$ appears only on the right ($p$ times), so \vspace{-1mm}
\begin{longtable}[c]{l}
$\mathring{\mathcal{M}}_{k-1}=\emptyset$ \;\;and\;\; $\mathring{\mathcal{M}}_k= \{e_\sigma,\; e_{ab_1}\!\ldots\!e_{ab_p} e_{b_1b_p}\!\ldots\!e_{b_{p\!-\!1}b_p}\}$ \;\;and\;\; \\
$\mathring{\mathcal{M}}_{k+1}= \{e_\sigma,\; e_{\{x\}}e_{ab_1}\!\ldots\!e_{ab_p} e_{b_1b_p}\!\ldots\!e_{b_{p\!-\!1}b_p},\; e_{ab_1}\!\ldots\!e_{ab_p} e_{b_1c}\!\ldots\!e_{b_pc}\}$,
\end{longtable}\vspace{-2mm}
\noindent where $a\!<\!b_1\!<\!\ldots\!<\!b_p\!<\!c$. Now obviously $\mathring{\partial}_k\!=\!0$, but unfortunately $\mathring{\partial}_{k+1}\!\neq\!0$, since \smash{$\mathring{\partial}\big(\bigwedge_{i\in[p]}\!e_{ab_i}e_{b_ic}\big)=\sum_{j\in[p]}\!\pm e_{ac}\!\bigwedge_{i\in[p]\setminus\{j\}}\!e_{ab_i}e_{b_ic}$}.
Denote $b'\!=\!b\!+\!1$ and define $$\mathcal{M}'=\Big\{
\begin{smallmatrix} e_{ab_1}\ldots e_{ab'_{i\!-\!1}}\!e_{b'_{i\!-\!1}b_i}\ldots e_{ab_p} e_{b_1b_p}\ldots e_{b_{p\!-\!1}b_p}\\[-2pt] \hspace{-50pt}\downarrow\\[-1pt]
                    v\,=\,e_{ab_1}\ldots e_{ab_i}\ldots e_{ab_p}e_{b_1b_p}\ldots e_{b_{p\!-\!1}b_p}\hspace{10pt}\end{smallmatrix}\!\!;
\begin{smallmatrix}i\text{ is minimal such that}\\ i<p\text{ and }b'_{i\!-\!1}\neq b_i\end{smallmatrix}\Big\},$$
i.e. take the first nonconsecutive index $b_i\!\in\!\epsilon(v)$ and split it to get a consecutive one. Because of $e_{ab_p}$, for every $u\!\overset{\scriptscriptstyle\mathcal{M}'}{\to}\!v$ there is exactly one $u\!\to\!v'$ (so zig-zags are unique). Without the condition $i\!<\!p$, there would exist cycles in $\Gamma_{C_\ast}^{\mathcal{M}'}$\!, e.g. if $n\!=\!k\!=\!3$, $p\!=\!2$,\vspace{-2mm} $$\xymatrix@R=0.8mm@C=11mm{e_{12}e_{14}e_{24}&e_{12}e_{13}e_{24}e_{34}\ar[l]|-{\scriptscriptstyle\mathcal{M}'}\ar[dl]\\
                          e_{13}e_{14}e_{34}&e_{12}e_{14}e_{23}e_{34}.\ar[l]|-{\scriptscriptstyle\mathcal{M}'}\ar[ul]}\vspace{-2mm}$$
Our set $\mathcal{M}'$ is a Morse matching, with surviving vertices\vspace{-1mm}
$$\mathring{\mathcal{M}}'_k\!=\!\big\{e_\sigma,\; e_{a,a+1}e_{a,a+2}\ldots e_{a,a+p-\!1}e_{a,b_p}e_{a+1,b_p}e_{a+2,b_p}\ldots e_{a+p-\!1,b_p}\big\}\text{ and }\mathring{\partial}'_{k+1}\!=\!0.\vspace{-1mm}$$ Since $a,b_p\!\in\![n]$ and there are $p\!-\!1$ elements inbetween, there are $\binom{n\!-\!p+1}{2}$ choices.
\end{proof}\vspace{1mm}

\vspace{4mm}
\section{Strictly triangular matrices $\frak{nil}_n(R)$}\label{section4}
In the digraph for $\frak{nil}_n$, there are many more isolated vertices, but it turns out that a subcomplex determines the homology of $\frak{nil}_n$ over a field.\vspace{2mm}

\subsection{Shifted binomials} There are some patterns in the table from \ref{section3}, with which the modules $\Z_{p^m}$ appear. Take for example the fourth column. By the Universal Coefficient Theorem, $\dim\,H_k(\frak{sol}_4;\Z_2)=1,4,6,7,15,26,24,11,2,0,0$. If we subtract $\binom{4}{k}\!=\!1,4,6,4,1$, we get $0,0,0,3,14,26,24,11,2,0,0$. If we subtract $3\binom{4}{k}$ shifted by $3$, we get $0,0,0,0,2,8,12,8,2,0,0$. If we subtract $2\binom{4}{k}$ shifted by $4$, we get $0,0,0,0,0,0,0,0,0,0,0$. Thus $\dim\,H_k(\frak{sol}_4;\Z_2)= \binom{4}{k}\!+\!3\binom{4}{k\!-\!3}\!+\!2\binom{4}{k\!-\!4}$. Doing this for all small $n$ and $p$ produces a table for $\dim\,H_k(\frak{sol}_n;\Z_p)$:{\small
\begin{longtable}[l]{r|lll}
$n \backslash p$ &$2$\\ \hline
$3$ &$\binom{n}{k}\!+\!\binom{n}{k\!-\!3}$,\\
$4$ &$\binom{n}{k}\!+\!3\binom{n}{k\!-\!3}\!+\!2\binom{n}{k\!-\!4}$,\\
$5$ &$\binom{n}{k}\!+\!6\binom{n}{k\!-\!3}\!+\!5\binom{n}{k\!-\!4}\!+\!5\binom{n}{k\!-\!6}\!+\!6\binom{n}{k\!-\!7}\!+\!\binom{n}{k\!-\!10}$,\\
$6$ &$\binom{n}{k}\!+\! 10\binom{n}{k\!-\!3}\!+\! 9\binom{n}{k\!-\!4}\!+\! 30\binom{n}{k\!-\!6}\!+\! 61\binom{n}{k\!-\!7}\!+\! 30\binom{n}{k\!-\!8}\!+\! 15\binom{n}{k\!-\!10}\!+\!
  19\binom{n}{k\!-\!11}\!+\! 5\binom{n}{k\!-\!12}$,
\end{longtable}}
{\small
\begin{longtable}[l]{r|lll}
$n \backslash p$ &$3$ &$5$ \\ \hline
$4$ &$\binom{n}{k}\!+\!\binom{n}{k\!-\!5}$,\\
$5$ &$\binom{n}{k}\!+\!3\binom{n}{k\!-\!5}\!+\!\binom{n}{k\!-\!6}\!+\!\binom{n}{k\!-\!8}$,\\
$6$ &$\binom{n}{k}\!+\!6\binom{n}{k\!-\!5}\!+\!3\binom{n}{k\!-\!6}\!+\!6\binom{n}{k\!-\!8}\!+\!4\binom{n}{k\!-\!9}$, &$\binom{n}{k}\!+\!\binom{n}{k\!-\!9}$.
\end{longtable}}\vspace{2mm}

\subsection{Tensor product} The above observation is not a coincidence. Given a prime $p$ and $\frak{g}\!\in\!\{\frak{nil}_n,\frak{sol}_n\}$, let $C_{\ast,p}(\frak{g};R)$ be the chain subcomplex of $C_{\ast}(\frak{g};R)$ spanned by vertices $\{v;\, \forall x\!\in\!\epsilon(v)\!: r_x\!-\!s_x\!\in\!p\Z\}$. Thus for $p\!=\!2$ our $C_{\ast,p}(\frak{nil}_n)$ is a subcomplex, spanned by vertices in which every index appears an even number of times, e.g. $C_{1,p}\!=\!0$, $C_{2,p}\!=\!0$, $C_{3,p}\!=\!\langle e_{ab}e_{ac}e_{bc}\rangle$, $C_{4,p}\!=\!\langle e_{ab}e_{ac}e_{bd}e_{cd}, e_{ab}e_{ad}e_{bc}e_{cd}, e_{ac}e_{ad}e_{bc}e_{bd}\rangle$, and so on. Let $\frak{dgn}_n\!=\!\{\sum_x\!r_xe_{xx}\}\!\leq\!\frak{sol_n}$ be the Lie subalgebra of all diagonals.\vspace{2mm}

\begin{Prp}\label{4.tensor} If $\Z_p\!\leq\!R$, then $\mathring{C}_\ast(\frak{sol}_n;R)\cong C_{\ast,p}(\frak{nil}_n;R)\!\otimes\!C_\ast(\frak{dgn}_n;R)$ and
$$\textstyle{\dim H_k(\frak{sol}_n;\Z_p)\,=\! \sum_{i+j=k}\!\dim H_i\big(C_{\ast,p}(\frak{nil}_n;\Z_p)\big)\!\cdot\!{\textstyle\binom{n}{j}}}\vspace{-1mm}$$ and $\{p\text{-torsion of }H_\ast C_{\ast}(\frak{sol}_n;\Z)\} \!=\! \{p\text{-torsion of }H_\ast C_{\ast,p}(\frak{sol}_n;\Z)\}$.
\end{Prp}
\begin{proof} By \ref{3.Zp} we have $\mathring{\mathcal{M}}\!=\!\{v; \forall x\!\in\!\epsilon(v)\!:r_x\!-\!s_x\!\in\!p\Z\}$ and $\mathring{\partial}\!=\!\partial$. By definition, the tensor product of chain complexes $(C_\ast,\partial_\ast)$ and $(C'_\ast,\partial'_\ast)$ is given by $(C_\ast\!\otimes\!C'_\ast)_k\!=\!\bigoplus_{i+j=k}\!C_i\!\otimes\!C'_j$ and  $\partial\!\otimes\!\partial'(x\!\otimes\!x')\!=\! (\partial x)\!\otimes\!x'+(-1)^{|x|}x\!\otimes\!(\partial'\!x')$. Our $\frak{dgn}_n$ has trivial brackets, so $C_\ast(\frak{dgn}_n)$ has zero boundaries. In $\mathring{C}_\ast(\frak{sol}_n)$ the boundary of each vertex is a sum of all possible brackets of nondiagonal elements, because diagonals contribute the summands with weights as coefficients, which are $0$ over $\Z_p$. It follows that $\mathring{C}_\ast(\frak{sol}_n;R)\cong C_{\ast,p}(\frak{nil}_n;R)\!\otimes\!C_\ast(\frak{dgn}_n;R)$. Then K\"{u}nneth's theorem provides the formula for homology of the tensor product.
\par The statement that the $p$-torsion in the homology of the chain complex for $\frak{sol}_n(\Z)$ equals the $p$-torsion in the homology of $C_{\ast,p}(\frak{sol}_n;\Z)$ follows from the fact that $C_{\ast,p}$ is a direct summand of $C_\ast$ and the Universal Coefficient theorem: if the complement of $C_{\ast,p}$ contained any $p$-torsion, then tensoring with $\Z_p$ would imply that the complement contributes to $\dim H_k(\frak{sol}_n;\Z_p)$, a contradiction with $\mathring{\mathcal{M}}$.
\end{proof}\vspace{2mm}

\subsection{$p$-complex}\label{4.3} We've seen that chain complex $C_{\ast,p}(\frak{nil}_n;\Z_p)$ determines the homology of $\frak{sol}_n$ over $\Z_p$. We wish to find some patterns for the former. Computer experiments reveal the homology of $C_{k,p}(\frak{nil}_n;\Z)$ for small $n$ and $p\!=\!\rd{2},\gn{3},\bl{5}$:
\vspace{3pt}{\small
\begin{longtable}[c]{@{\hspace{0pt}}r@{\hspace{5pt}}|l@{\hspace{5pt}} l@{\hspace{5pt}} l@{\hspace{5pt}} l@{\hspace{5pt}} l@{\hspace{5pt}} l@{\hspace{5pt}}}
$k \backslash n$ &$2$&$3$&$4$&$5$&$6$&$7$\\ \hline
$1$  &&&         &                            &                                          &\\
$2$  &&&         &                            &                                          &\\
$3$  &&$\rd{\Z}$&$\rd{\Z^2\!\oplus\!\Z_2}$&$\rd{\Z^3\!\oplus\!\Z_2^3}$&$\rd{\Z^4\!\oplus\!\Z_2^6}$&$\rd{\Z^5\!\oplus\!\Z_2^{10}}$\\
$4$  &&&$\rd{\Z}$&$\rd{\Z^2}$                 &$\rd{\Z^3}$                               &$\rd{\Z^4}$\\
$5$  &&&$\gn{\Z}$&$\gn{\Z^2\!\oplus\!\Z_3}$   &$\gn{\Z^3\!\oplus\!\Z_3^3}$               &$\gn{\Z^4\!\oplus\!\Z_3^6}$\\
$6$  &&&         &$\rd{\Z^2\!\oplus\!\Z_2^3}$ &$\rd{\Z^5\!\oplus\!\Z_2^{24}\!\oplus\!\Z_4\!\oplus\!\Z_3^2}$ &$\rd{\Z^9\!\oplus\!\Z_2^{73}\!\oplus\!\Z_4^2\!\oplus\!\Z_3^4}$\\
$7$  &&&&$\rd{\Z^3}$   &$\rd{\Z^{10}\!\oplus\!\Z_2^{23}\!\oplus\!\Z_4^3\!\oplus\!\Z_3^3}$ &$\rd{\Z^{19}\!\oplus\!\Z_2^{84}\!\oplus\!\Z_4^{11}\!\oplus\!\Z_3^{10}}\!\oplus\!\gn{\Z_2^2}$\\
$8$  &&&&$\gn{\Z}$     &$\rd{\Z^4}\!\oplus\!\gn{\Z^3\!\oplus\!\Z_3^3}$ &$\rd{\Z^{10}\!\oplus\!\Z_2^5\!\oplus\!\Z_4^3\!\oplus\!\Z_3^5}\!\oplus\!\gn{\Z^5\!\oplus\!\Z_3^6}$\\
$9$  &&&&              &$\gn{\Z}\!\oplus\!\bl{\Z}$                &$\rd{\Z^5\!\oplus\!\Z_2^{72}\!\oplus\!\Z_4^3\!\oplus\!\Z_3^2}\!\oplus\!\gn{\Z^2}\!\oplus\!\bl{\Z^2\!\oplus\!\Z_5}$\\
$10$ &&&&$\rd{\Z}$ &$\rd{\Z^4\!\oplus\!\Z_2^9\!\oplus\!\Z_4^2\!\oplus\!\Z_3}\!\oplus\!\gn{\Z_2}$
     &$\rd{\Z^{19}\!\oplus\!\Z_2^{176}\!\oplus\!\Z_4^{18}\!\oplus\!\Z_3^6}\!\oplus\!\gn{\Z^2\!\oplus\!\Z_2^5\!\oplus\!\Z_3^7}$\\
$11$ &&&&          &$\rd{\Z^4\!\oplus\!\Z_2^4}$               &$\rd{\Z^{19}\!\oplus\!\Z_2^{72}\!\oplus\!\Z_4^3\!\oplus\!\Z_3^2}\!\oplus\!\gn{\Z\!\oplus\!\Z_2^6\!\oplus\!\Z_3}$\\
$12$ &&&&          &$\rd{\Z}$                                 &$\rd{\Z^5\!\oplus\!\Z_2^5\!\oplus\!\Z_4^3\!\oplus\!\Z_3^5}\!\oplus\!\gn{\Z_2^5}$\\
$13$ &&&&          &                                          &$\rd{\Z^{10}\!\oplus\!\Z_2^{84}\!\oplus\!\Z_4^{11}\!\oplus\!\Z_3^{10}}\!\oplus\!\gn{\Z^4\!\oplus\!\Z_3^7}$\\
$14$ &&&&          &                                          &$\rd{\Z^{19}\!\oplus\!\Z_2^{73}\!\oplus\!\Z_4^2\!\oplus\!\Z_3^4}\!\oplus\!\gn{\Z^2}$\\
$15$ &&&&          &                                          &$\rd{\Z^9}\!\oplus\!\gn{\Z_2^2}$\\
$16$ &&&&          &                                          &$\gn{\Z^2\!\oplus\!\Z_3}\!\oplus\!\bl{\Z}$\\
$17$ &&&&          &                                          &$\rd{\Z^4\!\oplus\!\Z_2^{10}}$\\
$18$ &&&&          &                                          &$\rd{\Z^5}$\\
$19$ &&&&          &                                          &$\gn{\Z}$\\
$20$ &&&&          &                                          &\\
$21$ &&&&          &                                          &$\rd{\Z}$\\
\end{longtable}}\vspace{2pt}
The presence of summand $\Z_3$ in $C_{\ast,2}$ and summand $\Z_2$ in $C_{\ast,3}$ shows that $\mathring{C}_{\ast,p}(\frak{sol}_n;\Z)\ncong C_\ast(\frak{dgn}_n;\Z)\!\otimes\!C_{\ast,p}(\frak{nil}_n;\Z)$. This also means that a matching on $C_{\ast,p}(\frak{nil}_n;\Z_p)$ cannot induce $\mathring{\partial}\!=\!0$ (because for $\frak{nil}_n$ all weights are $\pm1$, so digraph $\Gamma_{C_\ast}$ over $\Z$ or $\Z_p$ has the same edges, hence $\mathcal{M}$ is a matching over $\Z$ iff over $\Z_p$, but the presence of torsion in homology over $\Z$ implies nonzero differential). Thus we expect the computation to be tricky. It is difficult to find a recursion or a generating function or a closed formula for $H_k(C_{\ast,p}(\frak{nil}_n;\Z_p))$. The homology of the whole complex $C_\ast(\frak{nil}_n;\Z)$ is studied in the article \cite[4.8]{citearticleLampretVavpeticCPLA}.\vspace{2mm}

\vspace{4mm}
\section{Poset triangular matrices $\frak{gl}_n^\preceq(R)$}\label{section5}
For any partial ordering $\preceq$ on $[n]$, our Lie algebra $\frak{g}\!=\!\frak{gl}^\preceq_n(R)$ is a subalgebra of $\frak{gl}_n(R)$ that admits an $R$-module basis $\{e_{ij}; 1\!\leq\!i\!\preceq\!j\!\leq\!n\}$. This is a large class (parametrized by all finite posets) of \emph{solvable} Lie algebras, i.e. $\exists r\!: \frak{g}^{(r)}\!=\!0$ where the derived series is defined inductively by $\frak{g}^{(0)}\!=\!\frak{g}$ and $\frak{g}^{(r)}\!=\![\frak{g}^{(r\!-\!1)}\!,\frak{g}^{(r\!-\!1)}]$.\vspace{1mm}

A number of properties of $\frak{sol}_n$ generalize to $\frak{gl}^\preceq_n$.
\begin{Prp} $\{\text{isolated vertices of }\Gamma_{C_\ast}\}\!=\!\{e_\sigma; \sigma\!\subseteq\![n]\}$ over $R\!=\!\Z$.
\end{Prp}

If we restrict the ordering $\preceq$ from $[n]$ to $[n\!-\!1]$, we obtain the subalgebra $\frak{gl}^\preceq_{n\!-\!1}$.
\begin{Prp}\label{5.RowGrowth}$H_k(\frak{gl}^\preceq_n)\!\cong\! H_k(\frak{gl}^\preceq_{n\!-\!1})\!\oplus\!\ldots$ and $H^k(\frak{gl}^\preceq_n)\!\cong\! H^k(\frak{gl}^\preceq_{n\!-\!1})\!\times\!\ldots$.
\end{Prp}

We can define the set of edges $\mathcal{M}$ similarly as for $\frak{sol}_n$, by using diagonals.
\begin{Prp}\label{5.matching} $\mathcal{M}$ is a Morse matching, with $\mathring{\partial}\!=\!\partial$ and $$\mathring{\mathcal{M}}\!=\!\big\{v;\, \forall x\!\in\!\epsilon(v)\!:r_x\!\!-\!s_x\!\in\!R\!\setminus\!R^\times\big\}.$$
\end{Prp}
The proof is basically the same as in \ref{3.matching}. In fact, for any Lie algebra $\frak{g}\!\leq\!\frak{gl}_n(R)$ that admits an $R$-module basis $\mathcal{B}$ with $\{e_{ii};i\!\in\![n]\}\!\subseteq\! \mathcal{B} \!\subseteq\!\{e_{ij}; i,j\!\in\![n]\}$, the set $$\mathcal{M}\!=\!\Big\{\!
\begin{smallmatrix}e_{a_1x}\ldots e_{a_rx}e_{xx}          e_{xb_1}\!\ldots e_{xb_s} e_{c_1d_1}\!\ldots e_{c_td_t}e_\sigma\\[-2pt] \hspace{-45pt}\downarrow\\[-2pt]
                   e_{a_1x}\ldots e_{a_rx}\widehat{e}_{xx}e_{xb_1}\!\ldots e_{xb_s} e_{c_1d_1}\!\ldots e_{c_td_t}e_\sigma\end{smallmatrix}\!;
\begin{smallmatrix}x\in\epsilon(v)\text{ is minimal such }\\\text{that }r-s\text{ is a unit of }R \end{smallmatrix}\!\!\Big\}$$ is a Morse matching, where new boundary operators are the restrictions of old boundary operators and critical vertices are the wedges in which every index $x$ has noninvertible weight (=number of times $x$ appears on the right minus number of times $x$ appears on the left) in $R$, i.e. $$\mathring{\partial}\!=\!\partial~~~~\text{ and }~~~~\mathring{\mathcal{M}}\!=\!\big\{v;\, \forall x\!\in\!\epsilon(v)\!:r_x\!\!-\!s_x\!\in\!R\!\setminus\!R^\times\big\}.$$  However, $\mathring{\mathcal{M}}$ for general Lie algebras (such as $\frak{gl}_n(R)$) may be much larger than $\mathring{\mathcal{M}}$ for $\frak{gl}_n^\preceq$ (i.e. many vertices which don't contribute to homology may survive), but a decrease in size for the chain complex by a factor of $10$ or $10^2$ or $10^3$ (depending on how many integers of the base ring $R$ are invertible) is still quite beneficial. Another example where this matching is useful is $\mathfrak{so}_n(R)\!=\!\{\mathbf{a}\!\in\!\mathfrak{gl}_n(R); \mathbf{a}^t\!=\!-\mathbf{a}\}$ when $R$ has characteristic $2$, since it admits a basis $\{e'_{ab},e_{cc};\, a\!<\!b\}$ where $e'_{ab}\!=\!e_{ab}\!-\!e_{ba}$ and brackets are given by $[e'_{ab},e'_{cd}]\!=\! \delta_{bc}e'_{ad} \!+\!\delta_{ad}e'_{bc} \!-\!\delta_{bd}e'_{ac}\!-\! \delta_{ac}e'_{bd}$ and $[e'_{ab},e_{cc}]\!=\! \delta_{bc}e'_{ac}\!+\!\delta_{ac}e'_{bc}$, so the diagonals can be used to define $\mathcal{M}$.\vspace{1mm}

\begin{Rmk}
Actually, this matching works in an even more general context: for an arbitrary Lie algebra (not necessarily a subalgebra of $\mathfrak{gl}_n$) all we need is a module basis $\mathcal{B}$ and a subset $\{b_i\}\!\subseteq\!\mathcal{B}$ with the property that for every $b\!\in\!\mathcal{B}$ the bracket $[b_i,b]$ is either $0$ or $\beta_ib$ for some unit $\beta_i\!\in\!R^\times$. This principle is not limited to Lie algebras, but to (co)homology theories whose (co)chain complex consists of exterior powers. It is somewhat analogous to the property \cite[6.5.3, p.178]{citeWeibelIHA} and \cite[1.1.15, p.13]{citeLodayCH} of group and Hochschild homology (where tensor powers are used instead of exterior powers), which say that removing all copies of the identity element $1$ preserves the homology of the complex. Thus we call any such $\mathcal{M}$ the \emph{normalization matching}.
\par The most useful part for programmers is that we only need to remove columns and rows (very many of them) of boundary matrices $\partial_k$ to obtain new boundary matrices $\mathring{\partial}_k$. This is in stark contrast to (co)homology theories where the weights in the digraph of the corresponding (co)chain complex are only $0$ or $\pm1$ (such as simplicial or group or Hochschild (co)homology), since there the new boundary $\mathring{\partial}$ is almost never the restriction of the old boundary $\partial$.
\end{Rmk} \vspace{1mm}

\begin{Prp}\label{5.LowDegrees} If $\Z\!\cap\!R^\times\!=\!\{\pm1\}$, then $H_1(\frak{gl}_n^\preceq)\!\cong\!H^1\!(\frak{gl}_n^\preceq) \cong R^n$\!, $H_2(\frak{gl}_n^\preceq)\!\cong\!H^2\!(\frak{gl}_n^\preceq) \cong R^{\binom{n}{2}}$\!, $H_3(\frak{gl}_n^\preceq) \!\cong\! R^{\binom{n}{3}}\!\oplus\!(\frac{\scriptscriptstyle R}{\scriptscriptstyle 2R})^{m}$\!, where $m=|\{a,c\!\in\![n];\,\exists b\!: a\!\prec\!b\!\prec\!c\}|$.
\end{Prp}
This is obtained as in the proof of \ref{3.LowDegrees}: $H_3(\frak{g})\!\cong\! \frac{\langle e_{\{a,b,c\}},e_{ab}e_{ac}e_{bc}\rangle}{\langle e_{ax}e_{ac}e_{xc}-e_{ay}e_{ac}e_{yc}, 2e_{ab}e_{ac}e_{bc}\rangle}$, so the number of $\frac{R}{2R}$ summands equals the number of chains $a\!\prec\!b\!\prec\!c$ modulo the equivalence relation that identifies $a\!\prec\!x\!\prec\!c \,\sim\, a\!\prec\!y\!\prec\!c$, which is the number of pairs of elements in the poset that are comparable but not covering.\vspace{2mm}

\begin{Exp} A few drawings should illustrate the matter.\\[-20pt]{\small
\begin{longtable}[c]{@{\hspace{-5pt}}r@{\hspace{-3pt}}c@{\hspace{-3pt}} c@{\hspace{-3pt}} c@{\hspace{-3pt}} c@{\hspace{5pt}} c @{\hspace{-0pt}} r @{\hspace{0pt}}}
\begin{tikzpicture}\matrix (m)[matrix of nodes, row sep=12pt, column sep=0pt]{
          \\ Hasse diagram:\\ \\ \\};\end{tikzpicture}&
\begin{tikzpicture}\matrix (m) [matrix of math nodes, row sep=12pt, column sep=10pt]{
  & 4 & \\
2 &   & 3 \\
  & 1 &   \\ };
\draw(m-3-2)--(m-2-1)--(m-1-2); \draw (m-3-2)--(m-2-3)--(m-1-2);\end{tikzpicture}&
\begin{tikzpicture}\matrix (m) [matrix of math nodes, row sep=12pt, column sep=10pt]{
  & 5 & \\
2 & 3 & 4 \\
  & 1 &   \\ };
\draw(m-3-2)--(m-2-1)--(m-1-2); \draw (m-3-2)--(m-2-2)--(m-1-2); \draw (m-3-2)--(m-2-3)--(m-1-2);\end{tikzpicture}&
\begin{tikzpicture}\matrix (m) [matrix of math nodes, row sep=3pt, column sep=10pt]{
  & 5 &  \\
3 &   &  \\
  &   & 4\\
2 &   &  \\
  & 1 &  \\ };
\draw(m-5-2)--(m-4-1)--(m-2-1)--(m-1-2); \draw(m-5-2)--(m-3-3)--(m-1-2); \end{tikzpicture}&
\begin{tikzpicture}\matrix (m) [matrix of math nodes, row sep=6pt, column sep=10pt]{
  & 6 &  \\
3 &   & 5\\
2 &   & 4\\
  & 1 &  \\ };
\draw(m-4-2)--(m-3-1)--(m-2-1)--(m-1-2); \draw(m-4-2)--(m-3-3)--(m-2-3)--(m-1-2); \end{tikzpicture}&
\begin{tikzpicture}\matrix (m) [matrix of math nodes, row sep=4pt, column sep=10pt]{
  & 6 &  \\
4 &   &  \\
3 &   & 5\\
2 &   &  \\
  & 1 &  \\ };
\draw(m-5-2)--(m-4-1)--(m-3-1)--(m-2-1)--(m-1-2); \draw(m-5-2)--(m-3-3)--(m-1-2); \end{tikzpicture}&\\[-1mm]
$H_3(\frak{gl}_n^\preceq;\Z)$: &$\Z^4\!\oplus\!\Z_2$ &$\Z^{10}\!\oplus\!\Z_2$ &$\Z^{10}\!\oplus\!\Z_2^3$ &$\Z^{20}\!\oplus\!\Z_2^5$ &$\Z^{20}\!\oplus\!\Z_2^6$ & $\lozenge$\\
\end{longtable}}
\end{Exp}\vspace{-1mm}

\begin{Prp}\label{5.torsion} For every interval $[a,b]\!=\!\{a,x_1,\ldots,x_t\!=\!b\}$ in poset $([n],\preceq)$, the module $H_{2t\!-\!3}(\frak{gl}^\preceq_n;\Z)$ has a direct summand $\Z_t$.
\end{Prp}
\noindent Indeed, $v\!=\!e_{ax_1}\!\ldots\!e_{ax_t}e_{x_1x_t}\!\ldots\!e_{x_{t\!-\!1}x_t}$ is in $\Ker\,\partial$, and no splitting is possible (because all $x$ between $a$ and $b$ already appear in $v$), so vertex $v$ is adjacent in the digraph only to $e_{aa}v$ and $e_{bb}v$, and the weights of those edges are $w_a\!=\!-t$ and $w_b\!=\!t$.\vspace{1mm}

\begin{Thm}\label{5.Q} If $\Q\!\leq\!R$ or $\Z_p\!\leq\!R$ with prime $p\!\geq\!n$ or $p\!>\!\frac{k+1}{2}$, then $$H_k\big(\frak{gl}^\preceq_n(R)\big)\cong H^k\big(\frak{gl}^\preceq_n(R)\big)\cong R^{\binom{n}{k}}.$$
\end{Thm}
Using \ref{5.matching}, our arguments are the same as in \ref{3.Q}, we only need to remember that $\epsilon(v)\!\subseteq\![n]$ does not have the usual total order but $\preceq$ instead. If every index $x\!\in\!\epsilon(v)$ has weight $0$, then necessarily $v\!=\!e_\sigma$ is the wedge of diagonals, because maximal elements of $\epsilon(v)$ appear only on the right (i.e. $r_x\!>\!0\!=\!s_x$) and minimal elements appear only on the left (i.e. $s_x\!>\!0\!=\!r_x$).\vspace{1mm} 

\begin{Rmk} At first we thought that the above statement also holds when $p\geq\max_{a,b}|[a,b]|$, the size of the largest interval in the poset. We thought that the largest interval $[a,b]\!=\!\{a,x_1,\ldots,x_t\!=\!b\}$ induced the vertex $e_{ax_1}\!\ldots\!e_{ax_t}e_{x_1x_t}\!\ldots\!e_{x_{t\!-\!1}x_t}$ that would give the largest torsion $\Z_t$. This turned out to be false. Consider the following two posets, specified by their Hasse diagrams:\vspace{-3mm}
$$\begin{tikzpicture}\matrix (m) [matrix of math nodes, row sep=20pt, column sep=15pt]{
4&5&6\\
1&2&3\\ };
\draw(m-1-1)--(m-2-1); \draw(m-1-1)--(m-2-2); \draw(m-1-1)--(m-2-3);
\draw(m-1-2)--(m-2-1); \draw(m-1-2)--(m-2-2); \draw(m-1-2)--(m-2-3);
\draw(m-1-3)--(m-2-1); \draw(m-1-3)--(m-2-2); \draw(m-1-3)--(m-2-3);\end{tikzpicture}\hspace{10mm}
\begin{tikzpicture}\matrix (m) [matrix of math nodes, row sep=5pt, column sep=11pt]{
6&7&&8&9\\
&&5&&\\
1&2&&3&4\\ };
\draw(m-3-1)--(m-2-3); \draw(m-3-2)--(m-2-3); \draw(m-3-4)--(m-2-3); \draw(m-3-5)--(m-2-3);
\draw(m-1-1)--(m-2-3); \draw(m-1-2)--(m-2-3); \draw(m-1-4)--(m-2-3); \draw(m-1-5)--(m-2-3); \end{tikzpicture}\vspace{-3mm}$$
\noindent For left poset, let $v\!=\!e_{14}e_{15}e_{16}e_{24}e_{25}e_{26}e_{34}e_{35}e_{36}$ be the wedge of all nondiagonals. Then we have $\partial(v)\!=\!0$ and $\partial(e_{\{i\}}v)\!=\!\pm3v$ for all $i$, so $v$ generates a summand $\Z_3$, even though the largest interval has $2$ elements. The analogous statement also holds for the right poset, which gives $\Z_5$. For any prime power $m\!=\!p^l$, a generalization (full bipartite graph on $2m$ vertices) shows that in a poset $\preceq$, the largest interval may have only $2$ elements, but $H_\ast(\frak{gl}_n^\preceq;\Z)$ contains a summand $\Z_m$.
\par It remains an open problem to determine the largest torsion appearing in $H_\ast(\frak{gl}_n^\preceq)$ (as a function of $\preceq$), or to at least find some meaningful bounds on the torsion.
\end{Rmk}\vspace{1mm}

\begin{Thm}\label{5.cup} If $\Q\!\leq\!R$ or $\Z_p\!\leq\!R$ with prime $p\!\geq\!n$, then $$H^\ast\big(\frak{gl}^\preceq_n(R)\big)\cong \Lambda_R[x_1,\ldots,x_n]$$
as graded algebras, where $x_i$ has degree $1$ and corresponds to matrix $e_{ii}$. Assuming $\Z_p\!\leq\!R$ with $p\!=\!n\!-\!1$, if $([n],\preceq)$ has a least element $a$ and greatest element $b$, then $$H^\ast\big(\frak{gl}^\preceq_n(R)\big)\cong \Lambda_R[x_1,\ldots,x_n,y]$$
as graded algebras, where $y$ has degree $2p-\!1$ and corresponds to wedge $e_{ab}\bigwedge_x\!e_{ax}e_{xb}$, but if $([n],\preceq)$ is not bounded, then $H^\ast\big(\frak{gl}^\preceq_n(R)\big)\cong \Lambda_R[x_1,\ldots,x_n]$.
\end{Thm}
\begin{proof} By \ref{5.matching}, for $p\!\geq\!n$ the critical vertices $\mathring{\mathcal{M}}$ span a chain subcomplex of $C_\ast$, which is a direct summand whose complement is contractible. Thus under the assumptions, the inclusion of diagonals (which is a Lie algebra morphism) $\iota\!: \frak{dgn}_n\!\longrightarrow\!\frak{gl}_n^\preceq$ induces an isomorphism on all (co)homology modules, so the morphism of graded algebras $\iota^\ast\!: H^\ast(\frak{gl}_n^\preceq)\!\longrightarrow\!H^\ast(\frak{dgn}_n)$ is bijective. It remains to show that $H^\ast(\frak{dgn}_n)$ is the exterior polynomial algebra. Since in $\mathfrak{dgn}_n$ all brackets are zero and it has an $R$-module basis $e_{11},\ldots,e_{nn}$, the module $H^k(\frak{g})$ has a basis $\{\chi_\sigma; \sigma\!\subseteq\![n],|\sigma|\!=\!k\}$ where map $\chi_{\sigma}$ sends the basis vector $e_\sigma$ to $1$ and all other basis vectors $e_\tau$ to $0$. By the formula from \ref{2.formulation} for the cup product, we have $\chi_\sigma\!\!\smile\!\chi_\tau=
{\scriptstyle\big\{\!\begin{smallmatrix}   \pm\chi_{\sigma\cup\tau} & \text{if } \sigma\cap\tau=0\\[-1pt]
                0 & \text{if } \sigma\cap\tau\neq0\end{smallmatrix}}$
where the sign is determined by the number of transpositions that are required to order $\sigma\!\cup\!\tau$. This is precisely the multiplication in $\Lambda_R[x_1,\ldots,x_n]$.
\par Now assume that $p\!\geq\!n\!-\!1$. Critical vertices are wedges in which every index has weight $0$ or $p$ (only possible for a greatest element) or $-p$ (only possible for a least element). If the poset is not bounded, then critical vertices are precisely $\{e_\sigma;\sigma\!\subseteq\![n]\}$, and the result follows as above. But if our poset equals the interval $[a,b]$, then $\mathring{\mathcal{M}}\!=\! \{e_\sigma,e'_\sigma; \sigma\!\subseteq\![n]\}$ and $\mathring{\partial}\!=\!0$, where $e'_\sigma$ is the wedge of $e_\sigma$ and the vertex from \ref{5.torsion}. Thus $H^\ast(\frak{gl}_n^\preceq)$ has a dual module basis $\{\chi_\sigma,\chi'_\sigma; \sigma\!\subseteq\![n]\}$. Furthermore, $\chi_\sigma\!\!\smile\!\chi'_\tau=
{\scriptstyle\big\{\!\begin{smallmatrix}   \pm\chi'_{\sigma\cup\tau} & \text{if } \sigma\cap\tau=0\\[-1pt]
                0 & \text{if } \sigma\cap\tau\neq0\end{smallmatrix}}$ and $\chi'_\sigma\!\!\smile\!\chi'_\tau=0$, so $\chi'_{\emptyset}$ is the additional generator $y$.
\end{proof}\vspace{1mm}

\vspace{4mm}
\section{Poset strictly triangular matrices $\frak{gl}_n^\prec(R)$}\label{section6}
Analogously as in \ref{4.3}, we can define a chain subcomplex, the $p$-complex for $\frak{gl}^\preceq_n(R)$ and $\frak{gl}^\prec_n(R)$, which begets:
\begin{Prp}\label{6.tensor} If $\Z_p\!\leq\!R$, then $\mathring{C}_\ast(\frak{gl}^\preceq_n;R)\cong C_{\ast,p}(\frak{gl}^\prec_n;R)\!\otimes\!C_\ast(\frak{dgn}_n;R)$ and
$$\textstyle{\dim H_k(\frak{gl}^\preceq_n;\Z_p)=\! \sum_{i+j=k}\!\dim H_i\big(C_{\ast,p}(\frak{gl}^\prec_n;\Z_p)\big)\!\cdot\!{\textstyle\binom{n}{j}}}\vspace{-1mm}$$ and $\{p\text{-torsion of }H_\ast C_{\ast}(\frak{gl}^\preceq_n;\Z)\} \!=\! \{p\text{-torsion of }H_\ast C_{\ast,p}(\frak{gl}^\preceq_n;\Z)\}$.
\end{Prp}
The homology of Lie algebra $\frak{gl}^\prec_m(\Z)$ for small $m\!=\!2^n$ with respect to the poset of all subsets $(2^{[n]},\subseteq)$ has been computed in \cite[p.203]{citeJoswigTakayamaAGSS}. For general posets, some work on $H_\ast(\frak{gl}^\prec_n(\mathbb{C}))$ has been done in \cite{citearticleHozoIPHLAH}.\vspace{1mm}

\begin{Rmk} The homology of $C_{\ast,p}(\frak{gl}_n^\preceq;\Z)$ can also contain $q$-torsion for $q\!\neq\!p$. Indeed, for every interval in $([n],\preceq)$ of size $qp\!+\!1$, the wedge from \ref{5.torsion} has weights in $p\Z$, so it is critical and it contributes $\Z_{pq}\cong\Z_p\!\oplus\!\Z_q$.
\end{Rmk}

\vspace{4mm}
\section{Afterword}\label{section7}

\subsection{Conclusion} We saw that there are easily definable matchings on difficult chain complexes (using matrix units), that give surprising insights into the homology table and also the structure of the chain complex. The benefits are theoretical as well as computational. For instance, the first five columns of the table in \ref{section3} for $\frak{sol}_n$ were computed by brute force, but the sixth required the use of \ref{4.tensor}. For later columns, the use of that proposition is key for efficient computations.
\par Over ring $\Z$, the set of critical vertices is smaller than the set of all vertices by a factor of about $10$ (because the integer ring has many elements and only two units), but over  $\Z_p$ the gain is much bigger (since only one element is a nonunit), especially for large $p$. E.g. for the first six columns of the table \ref{3.table}, over $\Z_2$ we have $\tfrac{|\mathrm{rank}\,C_\ast|}{|\mathrm{rank}\,\mathring{C}_\ast|}\doteq10^2$, over $\Z_3$ we get $\tfrac{|\mathrm{rank}\,C_\ast|}{|\mathrm{rank}\,\mathring{C}_\ast|}\doteq10^3$, and over $\Z_5$ even $\tfrac{|\mathrm{rank}\,C_\ast|}{|\mathrm{rank}\,\mathring{C}_\ast|}\doteq10^4$.
\par For $\frak{gl}_n^\preceq$ the gain is even higher than for $\frak{sol}_n$: if both have the same number of basis elements $e_{ab}$, then the former has more diagonals $e_{aa}$ that give strict conditions for the critical vertices, so $\mathring{\mathcal{M}}$ for the latter is larger.
\par We expect these kinds of arguments to be applicable for many other Lie algebra families, as well as other objects of various (co)homology theories. Nice applications of \ref{6.tensor}, with more complicated calculations and explicit generating functions and presentations of cohomology algebras, can be found in \cite[5.1--6.6]{citearticleLampretVavpeticCPLA} and \cite{citearticleLampretVavpeticTTLAS}. \vspace{1mm}

\subsection{Computations} Let us share a few words of advice on how the table \ref{3.table} (or the homology of any Lie algebra $\frak{g}$ with basis $\{e_{ii};\, i\!\in\![n]\}\!\subseteq\!\mathcal{B}\!\subseteq\!\{e_{ij};\, i,j\!\in\![n]\}$) can be efficiently calculated. As we've seen in \ref{3.matching}, to any wedge $v\!\in\!\Lambda^{\!k}\frak{g}$ we associate a \emph{weight vector} $w_v\!=\!(w_1,\ldots,w_n)\!\in\!\Z^n$\!, where $w_i\!=\!r_i\!-\!s_i$ and $r_i$ is the number of times $i$ appears as a right index in $v$ (i.e. $r_i\!=\!|\{a\!\in\![n]; e_{ai}\!\in\!v\}|$) and $s_i$ is the number of times $i$ appears as a left index in $v$ (i.e. $s_i\!=\!|\{b\!\in\![n]; e_{ib}\!\in\!v\}|$). For any $w\!\in\!\Z^n$, denote by $[w]\!=\!\big\langle v\!\in\!\Lambda^{\!k}\frak{g};\, w_v\!=\!w\big\rangle$ the chain subcomplex spanned by all wedges with weight $w$. Since $[e_{ab},e_{bc}]\!=\!e_{ac}$, the boundary $\partial$ preserves weight vectors, hence digraph $\Gamma_{C_\ast}$\! is a disjoint union of all basis elements of $[w]$, and the chain complex $C_\ast$ is a direct sum of all $[w]$. Thus $H_k(\frak{g})\cong\bigoplus_{w}\!H_k([w])$ and $H^k(\frak{g})\cong\prod_{w}\!H^k([w])$.
\par If $[w]$ is nonempty, then $\forall i\!: i\!-\!n\!\leq\! w_i \!\leq\! i\!-\!1$ and $\sum_{i=1}^nw_i\!=\!0$, so the list of possible $w$'s is narrowed a bit. Next, using \ref{6.tensor} we compute $p$-torsion, for every prime $p\!<\!n$. For this, it suffices to concentrate just on all $w\!\in\!(p\Z)^n$\!, a huge reduction. There are some relations between the complexes $[w]$ which further reduce the complexity, for instance $[0,\ldots]\!\cong\![\ldots]\!\cong\![\ldots,0]$, so all the results from $\frak{sol}_{n\!-\!1}$ are used for $\frak{sol}_n$, and $[w_1,\ldots,w_n]\!\cong\![-w_n,\ldots,-w_1]$ via $e_{i,j}\!\mapsto\!e_{n-j,n-i}$, etc.
\par We used \textsc{Mathematica} to recursively compute the basis elements and construct boundary matrices of $[w]$. Then using \textsc{GAP} and \textsc{SAGE}, we calculated the invariant factors of those sparse matrices, and merged them into the homology table. For $\frak{sol}_6$ there were $141\!+\!20\!+\!2$ complexes for $p\!=\!2,3,5$  (some of which were empty), and the largest matrix was of size $50\!\times\!32$, instead of $10^6\!\times\!10^6$ in the brute-force approach; on an old home computer, the calculation lasted $3$ seconds.\vspace{1mm}

\subsection{Acknowledgment} This research was supported by the Slovenian Research Agency grants P1-0292-0101, J1-5435-0101, J1-6721-0101, BI-US/12-14-001.
\par We wish to thank Luka Stopar for lending us the use of his superior computer for the more demanding calculations, and some programming assistance.

\end{document}